\documentclass[twoside,11pt,reqno]{amsart}
\usepackage{tikz-cd}
\usepackage{amsmath,amssymb,amscd,mathrsfs,epic,wasysym,latexsym,tikz,mathrsfs,cite,hyperref}
\usepackage{stmaryrd}
\usepackage{pb-diagram}
\usepackage[matrix,arrow]{xy}
\usepackage{ytableau}

\usepackage[margin=1.5in]{geometry}

\makeatletter

\numberwithin{equation}{section}
\newtheorem{Lemma}[equation]{Lemma}
\newtheorem{Theorem}[equation]{Theorem}
\newtheorem{Corollary}[equation]{Corollary}
\theoremstyle{definition}  %% makes all of the theorem environments which follow appear in \rm
\newtheorem{Definition}[equation]{Definition}
\newtheorem{Remark}[equation]{Remark}

\newtheorem{Example}[equation]{Example}

\newcommand{\NEarrow}{\mathbin{\rotatebox[origin=c]{45}{$\Rightarrow$}}}
\newcommand{\SEarrow}{\mathbin{\rotatebox[origin=c]{-45}{$\Rightarrow$}}}
\renewcommand{\UParrow}{\mathbin{\rotatebox[origin=c]{90}{$\Rightarrow$}}}
\renewcommand{\DOWNarrow}{\mathbin{\rotatebox[origin=c]{-90}{$\Rightarrow$}}}

\let\<\langle
\let\>\rangle

\newcommand\Comment[2][\relax]{\space\par\medskip\noindent%
   \fbox{\begin{minipage}{\textwidth}\textbf{Comment\ifx\relax#1\else---#1\fi}\newline%
        #2\end{minipage}}\medskip
}

\def\b1{\text{\boldmath$1$}}

\def\bw{\text{\boldmath$w$}}

\newcommand{\Z}{\mathbb{Z}}

\def\phi{{\varphi}}

\newcommand{\la}{\lambda}

\newcommand{\ZZ}{{\mathbb Z}}
\newcommand{\NN}{{\mathbb N}}

\def\col{{\mathtt{c}}}

\newcommand{\Nodes}{\mathsf N}

\def\b{\mathfrak{b}}
\def\k{\Bbbk}

\def\T{{\mathtt T}}
\def\L{{\mathtt L}}
\def\R{{\mathtt R}}
\def\U{{\mathtt U}}

\newcommand{\Xlet}{{\mathscr X}}
\newcommand{\Ylet}{{\mathscr Y}}

{\catcode`\|=\active
  \gdef\set#1{\mathinner{\lbrace\,{\mathcode`\|"8000%
  \let|\midvert #1}\,\rbrace}}
}
\def\midvert{\egroup\mid\bgroup}

\colorlet{darkgreen}{green!50!black}
\tikzset{dots/.style={very thick,loosely dotted},
         greendot/.style={fill,circle,color=darkgreen,inner sep=1.5pt,outer sep=0},
         blackdot/.style={fill,circle,color=black,inner sep=1.5pt,outer sep=0},
         graydot/.style={fill,circle,color=gray,inner sep=1.1pt,outer sep=0}
}
\def\greendot(#1,#2){\node[greendot] at(#1,#2){}}
\def\blackdot(#1,#2){\node[blackdot] at(#1,#2){}}
\def\graydot(#1,#2){\node[graydot] at(#1,#2){}}

\newenvironment{braid}{% sets defaults for the braid diagrams
  \begin{tikzpicture}[baseline=6mm,black,line width=1pt, scale=0.32,
                      draw/.append style={rounded corners},
                      every node/.append style={font=\fontsize{5}{5}\selectfont}]%
  }{\end{tikzpicture}
}

\def\Grid(#1,#2){%  draws a coordinate grid inside a braid diagram
  \draw[very thin,gray,step=2mm] (0,0)grid(#1,#2);
  \draw[very thin,darkgreen,step=10mm] (0,0)grid(#1,#2);
}

\newcommand\Tableau[2][\relax]{
  \begin{tikzpicture}[scale=0.6,draw/.append style={thick,black}]
    \ifx\relax#1\relax%
    \else % shade the boxes in #1
      \foreach\box in {#1} { \filldraw[blue!30]\box+(-.5,-.5)rectangle++(.5,.5); }
    \fi
    \newcount\row\newcount\col
    \row=0
    \foreach \Row in {#2} {
       \col=1
       \foreach\k in \Row {
          \draw(\the\col,\the\row)+(-.5,-.5)rectangle++(.5,.5);
          \draw(\the\col,\the\row)node{\k};
          \global\advance\col by 1
       }
       \global\advance\row by -1
    }
  \end{tikzpicture}
}

\newcommand\YoungDiagram[2][\relax]{
  \begin{tikzpicture}[scale=0.6,draw/.append style={thick,black}]
    \ifx\relax#1\relax%
    \else % shade the boxes in #1
    \foreach\box in {#1} {
      \filldraw[blue!30]\box rectangle ++(1,1);
    }
    \fi
    \newcount\row
    \row=0
    \foreach \col in {#2} {
       \draw(1,\the\row)grid ++(\col,1);
       \global\advance\row by -1
    }
  \end{tikzpicture}
}

\begin{document}

\title[Super RSK correspondence with symmetry]{{\bf Super RSK correspondence with symmetry}}

\author{\sc Robert Muth}
\address{Dept. of Mathematics\\ Tarleton State University\\
Stephenville\\ TX 76401, USA}
\email{muth@tarleton.edu}

\begin{abstract}
Super RSK correspondence is a bijective correspondence between superbiwords and pairs of semistandard supertableaux. Such a bijection was given by Bonetti, Senato and Venezia, via an insertion algorithm closely related to Schensted insertion. Notably, the symmetry property satisfied by the classical RSK bijection holds only in special cases under this bijection. We present a new super RSK bijection, based on the mixed insertion process defined by Haiman, where the symmetry property holds in complete generality.
\end{abstract}

\maketitle

\section{Introduction}\label{introsec}
The work of Robinson \cite{Robinson} in 1938, and Schensted \cite{Schensted} in 1961, describes a bijection between permutations and pairs of same-shape standard tableaux, now known as the Robinson-Schensted (RS) correspondence. A key ingredient in the bijection is an algorithm called {\em Schensted insertion}. In 1970, Knuth showed that Schensted insertion could be adapted to a more general setting to achieve a bijection between two-line arrays of letters called `biwords' (which are in natural bijection with matrices of non-negative integers) and pairs of same-shape semistandard tableaux \cite{Knuth}. This bijection is known as the Robinson-Schensted-Knuth (RSK) correspondence.

A celebrated feature of the RSK correspondence is a certain symmetry property; namely, exchanging the rows of a biword (or, transposing the matrix from the matrix perspective) translates via the RSK correspondence to exchanging the positions of the associated pair of semistandard tableaux. This property, proven for the RS correspondence by Viennot \cite{Viennot} in 1977  and extended to the full RSK correspondence by Fulton \cite{Fulton} in 1997, is far from obvious from the workings of the RSK algorithm itself.

The RSK correspondence has applications in a variety of settings; of particular relevance for this paper is its application in representation theory and invariant theory, where it describes a bijection between various important bases for associative algebras and Lie algebras. We consider here the generalization of RSK correspondence to combinatorial objects associated with the representation theory and invariant theory of superalgebras. These `super' combinatorial objects are {\em restricted superbiwords} and {\em semistandard supertableaux}. In contrast with the classical situation, letters in restricted superbiwords can have even or odd parity, with repetition of mixed-parity biletters disallowed. Semistandard supertableaux are nondecreasing tableaux in which letters of even parity strictly increase down columns, and letters of odd parity strictly increase along rows. See for example \cite{CPT, DR, GRS, LNS, MZ} for a few instances of these combinatorial objects arising in the study of bases of superalgebras and their representations.

We prove that an adaptation of the mixed insertion algorithm defined by Haiman \cite{Haiman} can be used to define a `super-RSK' correspondence between restricted superbiwords and same-shape pairs of semistandard supertableaux. This correspondence fully generalizes the classical RSK correspondence, in the sense that classical RSK can be viewed as a specialization of super-RSK to the case of even-parity superbiwords, and the classical symmetry property described above holds for super-RSK in full generality.

Numerous variants of super-RSK correspondence exist in the literature. Most notably, Bonetti, Senato, and Venezia \cite{BSV} presented a different correspondence between the same sets of combinatorial objects considered in this paper. At the heart of their algorithm are dual insertion processes which are very much like the classical Schensted insertion process, in that insertion progresses linearly from one row to the next (or one column to the next), and the number of `bumps' in a given insertion is bounded by the number of rows (or the number of columns) in the Young diagram. By contrast, the Haiman insertion process utilized in this paper progresses in a less direct fashion, where the number of `bumps' in an insertion is bounded only by the number of nodes in the Young diagram. A more crucial difference between the two algorithms is the fact that  the super-RSK of \cite{BSV} does not have the symmetry property in general (see Example \ref{bigex}). La Scala, Nardozza and Senato describe \cite[Proposition 4.7]{LNS} a subset of superbiwords where symmetry is known to hold for the \cite{BSV} correspondence, but a complete description of such biwords is still an open problem.

Another variant of super-RSK correspondence appears in work by Shimozono and White \cite{SW}. Their algorithm is based around the same Haiman insertion algorithm as used here, but adapted to work with a different class of combinatorial objects: unrestricted superbiwords (repetitions of mixed biletters allowed), and supertableaux which are row-weak and column-strict with respect to {\em both} parities. They demonstrate a bijection between these objects, and prove that their correspondence generalizes the classical symmetry property as well. While the \cite{SW} algorithm generally yields different supertableaux from ours (see again Example \ref{bigex}), we note that they agree, crucially, in the special case of `standard' superbiwords---those with no repeated letters. This is a key ingredient in the proof of the symmetry of our super-RSK correspondence.

Now for a description of the structure of this paper. In \S\ref{prelimsec}, we set up basic notation and definitions of the relevant combinatorial objects. In \S\ref{insexsec}, we describe the \(\varepsilon\)-insertion process which drives the super-RSK algorithm, and prove some useful lemmas about the process. In \S\ref{behavsec}, we prove some bounds on the distribution of bumped nodes during the insertion process, which are necessary for the results in the subsequent section, and perhaps of independent interest in the study of tableau growth. In \S\ref{SRSKsec}, we define the super-RSK map `\(\textup{sRSK}\)' and prove the first main theorem of the paper, which appears as Theorem \ref{SuperRSK} in the text:\\

\noindent {\bf Theorem 1} (Super-RSK correspondence){\bf .} {\em The map \(\textup{sRSK}\) defines a bijection between restricted superbiwords and same-shape pairs of semistandard supertableaux.}\\

\noindent In \S\ref{symsec}, we prove some lemmas related to standardizing superbiwords, and prove the other main theorem of the paper, which appears as Corollary \ref{corsym} in the text:\\

\noindent {\bf Theorem 2} (Super-RSK symmetry){\bf .} {\em Under super-RSK correspondence, exchanging rows in the superbiword \(\bw\) is equivalent to exchanging the positions of the pair of supertableaux \(\textup{sRSK}(\bw)\).
}
\vspace{2mm}

\subsection{Acknowledgements} The author would like to thank Alexander Kleshchev, who originally suggested the topic of the paper as an approach to a representation theoretic problem, and provided helpful suggestions. The author would also like to thank Scott Cook for numerous fruitful discussions.

\section{Preliminaries}\label{prelimsec}

Since all combinatorial objects considered in this paper are \(\Z_2\)-colored, we will henceforth suppress the prefix `super' from most of our terminology.

\subsection{Alphabets}\label{alph}

An {\em alphabet} \(\Xlet\) is a set equipped with a parity function \(\Xlet \to \ZZ_2\),  \(x \mapsto \overline{x}\), and a total order \(<_{\Xlet}\). Elements of alphabets are called {\em letters}. We call \(x \in \Xlet\) {\em even} if \(\overline{x} = \overline{0}\) and {\em odd} if \(\overline{x} = \overline{1}\). Let \(\prec_{\Xlet}\) be the total order on \(\Xlet\) defined by
\begin{align*}
a \prec_{\Xlet} b \iff
\begin{cases}
\overline{a}=\overline{1}, \overline{b}=\overline{0}, \textup{ or}\\
\overline{a}=\overline{b} = \overline{0} \textup{ and } a <_\Xlet b, \textup{ or}\\
\overline{a}=\overline{b} = \overline{1} \textup{ and } a >_\Xlet b.\\
\end{cases}
\end{align*}
Note then that
\begin{align*}
a<_{\Xlet} b \implies
\begin{cases}
a \prec_{\Xlet} b & \textup{if } \bar{b} = \bar{0}\\
a \succ_{\Xlet} b & \textup{if } \bar{b} = \bar 1.
\end{cases}
\end{align*}
The {\em dual alphabet} \(\Xlet^*\) of an alphabet \(\Xlet\) has underlying set \(\{x^* \mid x \in \Xlet\}\), parity function defined by \(\overline{x^*} = \overline{x} + \overline{1}\),  and total order \(<_{\Xlet^*}\) defined so that 
\begin{align*}
x^* <_{\Xlet^*} y^* \iff x <_{\Xlet} y.
\end{align*}
It follows that \(a^* \prec_{\Xlet^*} b^*\) if and only if \(a \succ_{\Xlet} b\).

The {\em standardizing alphabet} \(\Xlet^\bullet\) of an alphabet \(\Xlet\) has underlying set \(\{x^{(i)} \mid x \in \Xlet, i \in \Z_{>0}\}\), parity function defined by \(\overline{x^{(i)}} = \overline{x}\), and total order \(<_{\Xlet^\bullet}\) defined so that
\begin{align*}
a^{(i)} <_{\Xlet^\bullet} b^{(j)} \iff
\begin{cases}
a<_{\Xlet} b, \textup{ or}\\
a=b \textup{ and } i<j.
\end{cases}
\end{align*}
Define the `forget superscripts' function \(\hat{\bullet}: \Xlet^{\bullet} \to \Xlet\) by \(x^{(i)} \mapsto x\).

Going forward, we will suppress the subscripts and write \(<\) or \(\prec\) when the underlying alphabet is clear from context.

\subsection{ Tableaux} 

We set
$\Nodes:= \Z_{>0}\times \Z_{>0}$ 
and refer to the elements of $\Nodes$
as {\em nodes}. 
%For $i\in I$, let $\Nodes^i=\{i\}\times \Z_{>0}\times \Z_{>0} \subset \Nodes$. 
Define a partial order $\leq$ on $\Nodes$ as follows: $(r,s)\leq(r',s')$ if and only if 
$r\le r'$ and $s\le s'$.  For \(u = (r,s) \in \Nodes\) we will write \(u':=(s,r) \in \Nodes\). 

For \(\varepsilon \in \Z_2\) and \(u=(r,s) \in \Nodes\), define
\begin{align*}
u_\varepsilon =
\begin{cases}
r & \varepsilon = \overline{0},\\
s& \varepsilon = \overline{1}.
\end{cases}
\end{align*}
and for \(i \in \ZZ_{>0}\) define
\begin{align*}
\Nodes(\varepsilon,i) = \{ u \in \Nodes \mid u_\varepsilon = i\}.
\end{align*}
I.e., \(\Nodes(\varepsilon,i)\) is the \(i\)th row of nodes if \(\varepsilon = \overline{0}\), and the \(i\)th column of nodes if \(\varepsilon=\overline{1}\). 

We write \(u \uparrow v\) if \(u_{\overline{0}} \geq v_{\overline{0}}\) and \(u_{\overline{1}} = v_{\overline{1}}\), and \(u\; \UParrow\; v\) if the inequality is strict. We write \(u \nearrow v\) if \(u_{\overline{0}} \geq v_{\overline{0}}\) and \(u_{\overline{1}} \leq v_{\overline{1}}\), and \(u\; \NEarrow\; v\) if both inequalities are strict. We similarly define the symbols \(\rightarrow, \Rightarrow, \searrow, \SEarrow, \downarrow, \DOWNarrow\;\).

For \(n \in \ZZ_{\geq 0}\), we say \(\lambda = (\lambda_1, \ldots, \lambda_n) \in \ZZ_{\geq 0}^n\) is a {\em partition of} \(n\), writing \(\lambda \vdash n\), if \(\lambda_1 \geq \cdots \geq \lambda_n\) and \(\sum \lambda_i = n\). Let \(\Lambda_+(n)\) be the set of all partitions of \(n\). The {\em Young diagram} of \(\lambda\) is 
\begin{align*}
[\lambda] = \{(r,s) \in \Nodes \mid s \leq \lambda_r\}.
\end{align*}
We say a node \(u\) of \([\lambda]\) is {\em removable} if \([\lambda] \backslash \{u\} = [\mu]\) for some partition \(\mu\). We say a node \(u \notin [\lambda]\) is {\em addable} if \([\lambda] \cup \{u\} = [\mu]\) for some partition \(\mu\). For a partition \(\lambda\), the {\em conjugate partition} \(\lambda'\) is defined such that \(u \in [\lambda']\) if and only if \(u' \in [\lambda]\).

An  (\(\Xlet,\la\)){\em-tableau} is a function \(\T:[\la]\to \Xlet\). If \(\T\) is an \((\Xlet, \la)\)-tableau, we write \(\textup{sh}(\T) = \la\). The {\em content} \(\textup{con}(\T)\) of an \((\Xlet,\la)\)-tableau \(\T\) is the multiset \(\{ \T(u) \mid u \in [\la]\}\).

An \((\Xlet,\la\))-tableau is {\em semistandard} if:
\begin{enumerate}
\item it is {\em non-decreasing}: \(\T(u) \leq_{\Xlet} \T(v)\) whenever \(u \leq v\).
\item it is {\em row-strict} with respect to odd letters: if \(\T(u) = \T(v)\) for \(u,v\) in the same row, then \(\overline{\T(u)}=\overline{0}\).
\item it is {\em column-strict} with respect to even letters: if \(\T(u) = \T(v)\) for \(u,v\) in the same column, then \(\overline{\T(u)}=\overline{1}\).
\end{enumerate}

A {\em standard} tableau is a semistandard tableau such that \(\T(u) \neq \T(v)\) for every \(u \neq v \in [\lambda]\). For an \((\Xlet, \lambda)\)-tableau \(\T\), define the {\em dual \((\Xlet^*, \lambda)\)-tableau} \(\T^*\) by \(\T^*(u):= \T(u)^*\), and define the {\em conjugate \((\Xlet, \lambda')\)-tableau} \(\T'\) by \(\T'(u):= \T(u')\). We write \(\T'^*:= (\T')^* = (\T^*)'\) for the {\em dual conjugate \((\Xlet^*, \la')\)-tableau}. The following lemmas are obvious.

\begin{Lemma}
The following are equivalent:
\begin{enumerate}
\item \(\T\) is a standard \((\Xlet,\lambda)\)-tableau.
\item \(\T'\) is a standard \((\Xlet,\lambda')\)-tableau.
\item \(\T^*\) is a standard \((\Xlet^*,\lambda)\)-tableau.
\end{enumerate}
\end{Lemma}

\begin{Lemma}\label{dualconj}
An \((\Xlet,\lambda)\)-tableau \(\T\) is semistandard if and only if the dual conjugate \(\T'^*\) is a semistandard \((\Xlet^*,\lambda')\)-tableau.
\end{Lemma}

\subsection{Standardizing tableaux}\label{stdtabsec}
Recalling the standardizing alphabet \(\Xlet^\bullet\) from \S\ref{alph}, for any \((\Xlet^\bullet, \lambda)\)-tableau \(\T\), define \(\hat \bullet(\T) := \hat \bullet \circ \T\). I.e., \(\hat \bullet(\T)\) is the tableau \(\T\) with superscripts deleted.
We say a standard \((\Xlet^\bullet, \la)\)-tableau \(\U\) is \(\bullet\)-{\em standard} provided 
\begin{enumerate}
\item \(\T:=\hat\bullet ( \U)\) is a semistandard \((\Xlet,\lambda)\)-tableau.
\item If \(u \nearrow v \in [\lambda]\) and \(\T(u) = \T(v)\), then \(\U(u)\prec_{\Xlet^\bullet} \U(v)\).
\end{enumerate}
We say then that \(\U\) is a \(\bullet\)-{\em standardization} of \(\T\).

\begin{Example}\label{firstex}
Take \(\Xlet = \{ \hat 1 < 1 < \hat 2 <  2 <  \hat 3 <  3 \}\), with odd elements indicated by carets. Let \(\lambda = (4,4,2)\). 
An \((\Xlet, \lambda)\)-tableau \(\T\),  its dual \(\T^*\), conjugate \(\T'\), and dual conjugate \(\T'^*\) are shown below.
\begin{align*}
\T=\ytableausetup{centertableaux}
\begin{ytableau}
 \hat{1} & 1 & 1 & \hat{2} \\
 \hat{1} & \hat{2} & 3 & 3\\
 \hat{1} & \hat{2}
\end{ytableau}
\hspace{6mm}
\T^*=\ytableausetup{centertableaux}
\begin{ytableau}
 \hat{1}^* & 1^* & 1^* & \hat{2}^* \\
 \hat{1}^* & \hat{2}^* & 3^* & 3^*\\
 \hat{1}^* & \hat{2}^*
\end{ytableau}
\hspace{6mm}
\T'=\ytableausetup{centertableaux}
\begin{ytableau}
\hat{1} & \hat{1} & \hat{1}\\
1 & \hat{2} & \hat{2} \\
1 & 3\\
\hat{2} & 3
\end{ytableau}
\hspace{6mm}
\T'^*=\ytableausetup{centertableaux}
\begin{ytableau}
\hat{1}^* & \hat{1}^* & \hat{1}^*\\
1^* & \hat{2}^* & \hat{2}^* \\
1 ^*& 3^*\\
\hat{2}^* & 3^*
\end{ytableau}
\end{align*}
Then \(\T\) is a semistandard \((\Xlet, \lambda)\)-tableau, and \(\T'^*\) is a semistandard \((\Xlet^*, \lambda')\)-tableau. In the standardizing alphabet \(\Xlet^\bullet\) we have
\begin{align*}
\hat 1^{(1)} < \hat 1^{(2)} < \hat 1^{(3)} < 1^{(1)} < 1^{(2)} < \hat 2^{(1)} < \hat 2^{(2)} < \hat 2^{(3)} < 3^{(1)} < 3^{(2)},
\end{align*}
and
\begin{align*}
 \hat 2^{(3)} \prec \hat 2^{(2)} \prec \hat 2^{(1)} \prec \hat 1^{(3)} \prec \hat 1^{(2)} \prec \hat 1^{(1)} \prec 1^{(1)} \prec 1^{(2)}   \prec 3^{(1)} \prec 3^{(2)},
\end{align*}
and so the \((\Xlet^\bullet, \la)\)-tableau
\begin{align*}
{\tt U}=
\ytableausetup{centertableaux,boxsize=1.7em}
\begin{ytableau}
\hat{1}^{(1)} & 1^{(1)} & 1^{(2)} & \hat{2}^{(1)} \\
 \hat{1}^{(2)} & \hat{2}^{(2)} & 3^{(1)} & 3^{(2)}\\
 \hat{1}^{(3)} & \hat{2}^{(3)}
\end{ytableau}
\end{align*}
is a \(\bullet\)-standardization of \(\T\).
\end{Example}

\section{ Insertion and Extraction}\label{insexsec}
It will be convenient in practice to formally extend the domain and range of an \((\Xlet, \lambda)\)-tableau \(\T\) to a function \(\T: \Nodes \to \Xlet \cup \{\infty\}\) by setting \(\T(u) = \infty\) for all \(u \notin [\lambda]\). We extend the order \(<\) on \(\Xlet\) to \(\Xlet \cup \{\infty\}\) by setting \(x < \infty\) for all \(x \in \Xlet\). We define the symbols \(\substack{\overline{0}\\ <}:=<\) and \(\substack{\overline{1} \\ <}:= \leq\). 

\subsection{Insertion}
Let \(\lambda \vdash n\), and assume \(\T\) is a semistandard \((\Xlet,\lambda)\)-tableau. Let \(\varepsilon \in \Z_2\) and \(x \in \Xlet\). From this data we construct an  \((\Xlet,\mu)\)-tableau \((\T \xleftarrow{\varepsilon} x)\), where \(\mu \vdash n+1\), via the method of \(\varepsilon\)-{\em insertion}.  \\

{\bf Algorithm for \(\varepsilon\)-insertion.}
\begin{enumerate}
\item[(1)] Set \(i:=1\), \(j:=1\), and \(x_1:=x\).

\item[(2)] Set \(b_j\) to be the smallest node in \(\Nodes(\varepsilon + \overline{x}_j,i)\) such that \(x_j \;\substack{\varepsilon \\ <}\; \T(b_j)\). 

\item[(3)] If \(\T(b_j) = \infty\), go to step (5). Otherwise, set \(j:=j+1\). 

\item[(4)] Set \(x_{j}:= \T(b_{j-1})\). Set \(i:= (b_{j-1})_{\varepsilon + \overline{x}_j} +1\) and go to step (2).

\item[(5)] Define \(\mu\) such that \([\mu] := [\la] \cup \{b_j\}\). Define \((\T \xleftarrow{\varepsilon} x)\) to be the  \((\Xlet,\mu)\)-tableau such that \((\T \xleftarrow{\varepsilon} x)(b_k) = x_k\) for all \(1\leq k \leq j\), and \((\T \xleftarrow{\varepsilon}x)(u) = \T(u)\) for all other nodes of \([\mu]\).
\\
\end{enumerate}

Assuming the process terminates when \(j=m\), we call \(b_1, \ldots, b_m\) the {\em bumped node sequence}, and call \(A(\T, \varepsilon, x):=b_m\)  the {\em added node}. We call \(x_1, \ldots, x_m\) the {\em bumped letter sequence}.

\begin{Remark}
Informally speaking, under \(\bar 0\)-insertion, bumped even letters are inserted in the next row down and bumped odd letters are inserted in the next column to the right. In \(\bar 1\)-insertion, this is reversed. The fuss over the differing comparisons \(\substack{\overline{0} \\ <}\) and \(\substack{\overline{1} \\ <}\) is needed to assure that semistandardness is maintained under \(\varepsilon\)-insertion, as will be shown in Lemma \ref{semisemi}.
\end{Remark}

\begin{Remark}\label{insertinLitH}
If \(\Xlet = \NN\), where \(<\) is the usual order on integers and every element is of even parity, then \(\overline{0}\)-insertion is {\em Schensted insertion} \cite{Schensted}. For general \(\Xlet\) and standard tableaux, \(\overline 0\)-insertion is {\em mixed insertion} as defined by Haiman (where odd letters are referred to as {\em circled}) \cite{Haiman}.
\end{Remark}

\begin{Remark}\label{insertinLitSW}
In \cite[\S3]{SW}, Shimozono and White define a process called {\em doubly-mixed insertion}, also adapted from from \cite{Haiman}, which is very similar to the \(\varepsilon\)-insertion presented in this paper; the processes are identical when applied to the subclass of standard tableaux. We note however that Shimozono and White use a different definition for semistandard tableaux---in their setup, semistandard tableaux are row-weak and column-strict with respect to both parities---so doubly-mixed insertion and \(\varepsilon\)-insertion differ substantially in the presence of repeated letters.
\end{Remark}

\begin{Example}
With \(\Xlet\) as in Example \ref{firstex} and \(\T\) as shown below, we have
\begin{align*}
\T=\ytableausetup{centertableaux}
\begin{ytableau}
 1 & \hat 2 & 2 \\
\hat 3 & 3 \\
\hat 3 
\end{ytableau}
\hspace{10mm}
(\T\xleftarrow{ \overline{0}} 1)=\ytableausetup{centertableaux}
\begin{ytableau}
 1 & 1 & \hat 2 \\
2 & \hat 3 \\
\hat 3 & 3
\end{ytableau}
\hspace{10mm}
(\T\xleftarrow{ \overline{1}} 1)=\ytableausetup{centertableaux}
\begin{ytableau}
 1 & 1 & 2 \\
\hat 2 &3 \\
\hat 3 \\
\hat 3
\end{ytableau}
\end{align*}

The bumped node sequence for \((\T\xleftarrow{ \overline{0}} 1)\) is
\begin{align*}
(1,2), (1,3), (2,1), (2,2), (3,2),
\end{align*}
and the bumped letter sequence is \(1,\hat 2, 2, \hat 3, 3\). The bumped node sequence for \((\T\xleftarrow{ \overline{1}} 1)\) is
\begin{align*}
(1,1), (1,2), (2,1), (3,1), (4,1),
\end{align*}
and the bumped letter sequence is \(1,1, \hat 2, \hat 3, \hat 3\).
\end{Example}

\subsection{Extraction}
Let \(\U\) be a semistandard \((\Xlet,\mu)\)-tableau, let \(u\) be a removable node of \([\mu]\), and set \(\la\) to be such that \([\lambda] = [\mu] \backslash \{u\}\). We define a \((\Xlet,\la)\)-tableau \((\U \xrightarrow{\varepsilon} u)\) by the method of \(\varepsilon\)-{\em extraction}.\\

{\bf Algorithm for \(\varepsilon\)-extraction.}
\begin{enumerate}
\item[(1)] Set \(j:=1\), \(c_1:=u\), and \(y_1:=\U(u)\). 
%If \(\varepsilon y_1 >0\), then go to step (2). Otherwise, go to step (3).
\item [(2)] Set \(i:= (c_j)_{\varepsilon + \overline{y}_j} -1\).
%Let \(i:=\ttr(c_j)-1\).
If \(i=0\), go to step (5).
\item [(3)] Set \(c_{j+1}\) to be the greatest node in \(\Nodes(\varepsilon+\overline{y}_j,i)\) such that \(y_j \;\substack{\varepsilon \\ >} \;\U(c_{j+1})\). %Otherwise, let \(c_{j+1}\) be the rightmost node in the \(i\)th row such that \(y_j \;\substack{\varepsilon \\ >} \;\U(c_{j+1})\). Go to step (4).
%\item [(3)] Let \(i:=\ttc(c_j)-1\). If \(i=0\), go to step (5). Otherwise, let \(c_{j+1}\) be the bottommost node in the \(i\)th column such that \(y_j \;\substack{\varepsilon \\ >} \;\U(c_{j+1})\). Go to step (4).
\item [(4)] Set \(j:=j+1\). Set \(y_{j}:=\U(c_{j})\).
% If \(\varepsilon y_{j} >0\), go to step (2). Otherwise, go to step (3).
Go to step (2).
\item [(5)] Define \((\U \xrightarrow{\varepsilon} u)\) to be the \(\lambda\)-tableau such that \((\U \xrightarrow{\varepsilon} u)(c_k) = y_{k-1}\) for all \(2 \leq k \leq j\), and \((\U \xrightarrow{\varepsilon} u)(v) = \U(v)\) for all other nodes of \([ \lambda]\).\\
\end{enumerate}

Assuming the process terminates when \(j=m\), we call \(y_1, \ldots, y_m\) the {\em unbumped letter sequence}., and define \(R(\U, \varepsilon, u):=y_m\) to be the {\em extracted letter}. We call \(c_1, \ldots, c_m\) the {\em unbumped node sequence}.

\begin{Example}
With \(\Xlet\) as in Example \ref{firstex} and \(\T\) as shown below, we have
\begin{align*}
\T=\ytableausetup{centertableaux}
\begin{ytableau}
 1 & \hat 2 & 2 \\
\hat 3 & 3 \\
\hat 3 
\end{ytableau}
\hspace{10mm}
(\T\xrightarrow{ \overline{0}} (3,1))=\ytableausetup{centertableaux}
\begin{ytableau}
 1 & \hat 2 & 2 \\
\hat 3 & 3 
\end{ytableau}
\hspace{10mm}
(\T\xrightarrow{ \overline{1}} (3,1))=\ytableausetup{centertableaux}
\begin{ytableau}
1 & 2 & \hat 3\\
\hat 3 & 3
\end{ytableau}
\end{align*}

The unbumped node sequence for \((\T\xrightarrow{ \overline{0}} (3,1))\) is only the node \((3,1)\), and the unbumped letter sequence is \(\hat 3\).

The unbumped node sequence for \((\T\xrightarrow{ \overline{1}} (3,1))\) is
\begin{align*}
(3,1), (2,1), (1,3), (1,2)
\end{align*}
and the unbumped letter sequence is \(\hat 3,\hat 3,  2, \hat 2\).
\end{Example}

\subsection{Some results on insertion and extraction} As noted in Remark \ref{insertinLitH}, \(\varepsilon\)-insertion is an adaptation of `mixed insertion' defined by Haiman \cite{Haiman}. Although Haiman works with {\em standard} tableaux, he remarks that his results may be extended to the semistandard case in a straightforward manner---this is outlined in \cite[\S1]{Haiman} and we take some pains to make the idea explicit in Lemma \ref{standsame}. Though some of the results in this subsection would follow from those in \cite{Haiman} and Lemma \ref{standsame}, we nevertheless include full proofs working in the general semistandard case, for the sake of clarity and self-containment, and since our notation and approach differ substantially from that of \cite{Haiman}.

The following two lemmas follow directly from definitions of the algorithms.
\begin{Lemma}\label{dualconjinsertion}
Let \(\T\) be a standard \((\Xlet, \la)\)-tableau, and assume \(x \in \Xlet \backslash \T\). Then
\begin{enumerate}
\item \((\T \xleftarrow{ \varepsilon} x)' = (\T' \xleftarrow{ \varepsilon +\overline{1}} x)\), and if \(b_1, \ldots, b_m\) is the bumped node sequence for \((\T \xleftarrow{ \varepsilon} x) \), then \(b_1', \ldots, b_m'\) is the bumped node sequence for \((\T' \xleftarrow{ \varepsilon +\overline{1}} x)\).
%\item \((\T \xleftarrow{ \varepsilon, <} x)' = (\T' \xleftarrow{ \overline{\varepsilon}, <} x)\), and if \(b_1, \ldots, b_m\) is the bumped node sequence for \((\T \xleftarrow{<,\varepsilon} x)\), then \(b_1', \ldots, b_m'\) is the bumped node sequence for \((\T' \xleftarrow{ \overline{\varepsilon}, <} x)\).
\item \((\T \xleftarrow{ \varepsilon} x)^* = (\T^* \xleftarrow{\varepsilon + \overline{1}} x^*)\),  and both insertions have the same bumped node sequence.
%\item \(\overline{(\T \xleftarrow{ \varepsilon, <} x)} = (\overline{\T} \xleftarrow{\overline{\varepsilon}, \overline{<}} \overline{x})\), and \((\T \xleftarrow{ \varepsilon, <} x)\) and \((\overline{\T} \xleftarrow{\overline{\varepsilon}, \overline{<}} \overline{x})\) have the same bumped node sequence.
\end{enumerate}
\end{Lemma}

\begin{Lemma}\label{dualconjsemi}
Let  \(\T\) be a semistandard \((\Xlet,\lambda)\)-tableau. Then
\begin{enumerate}
\item \((\T \xleftarrow{\varepsilon} x)'^* = (\T'^*   \xleftarrow{\varepsilon} x^*)\) for every \(x \in \Xlet\), and if \(b_1, \ldots, b_m\) is the bumped node sequence for \((\T \xleftarrow{\varepsilon} x)\), then \(b_1', \ldots, b_m'\) is the bumped node sequence for \((\T'^*   \xleftarrow{\varepsilon} x^*)\).
%\item \(\overline{(\T \xleftarrow{<,\varepsilon} x)}' = (\overline{\T}'   \xleftarrow{\overline{<},\varepsilon} \overline{x})\) for every \(x \in \colL\), and if \(b_1, \ldots, b_m\) is the bumped node sequence for \((\T \xleftarrow{<,\varepsilon} x)\), then \(b_1', \ldots, b_m'\) is the bumped node sequence for \((\overline{\T}   \xleftarrow{\overline{<},\varepsilon} \overline{x})\).
\item \((\T \xrightarrow{\varepsilon} u)'^* = (\T'^*   \xrightarrow{\varepsilon} u')\) for every removable node \(u \in [ \la]\), and if \(c_1, \ldots, c_m\) is the unbumped node sequence for \((\T \xrightarrow{\varepsilon} u)\), then \(c_1', \ldots, c_m'\) is the unbumped node sequence for \( (\T'^*   \xrightarrow{\varepsilon} u')\).
%\item \(\overline{(\T \xrightarrow{<,\varepsilon} u)}' = (\overline{\T}'   \xrightarrow{\overline{<},\varepsilon} u')\) for every removable node \(u \in \Y \la\), and if \(c_1, \ldots, c_m\) is the unbumped node sequence for \((\T \xleftarrow{<,\varepsilon} x)\), then \(c_1', \ldots, c_m'\) is the unbumped node sequence for \((\overline{\T}   \xleftarrow{\overline{<},\varepsilon} \overline{x})\).
\end{enumerate}
\end{Lemma}

\begin{Lemma}\label{inverses}
Let \(b_1, \ldots, b_m\) be the bumped node sequence, and let \(x_1, \ldots, x_m\) be the bumped letter sequence for the \(\varepsilon\)-insertion \((\T \xleftarrow{\varepsilon} x)\).  Assume \(i < j\). Then:
\begin{enumerate}
\item \(x_i \; \substack{ \varepsilon \\ <} \; x_j\)
\item \(\T(b_i) \; \substack{ \varepsilon \\ <} \; \T(b_j)\)
\item \((\T \xleftarrow{\varepsilon} x)(b_i)  \; \substack{ \varepsilon \\ <} \; (\T \xleftarrow{\varepsilon} x)(b_j)\)
%\item \(b_j \not \searrow b_i\) if \(\T\) is semistandard.
\item \(b_i \not \geq b_j\).
\end{enumerate}
\end{Lemma}
\begin{proof}
(i)--(iii) are obvious. For (iv), note that if \(\T\) is semistandard, then \(b_i \geq b_j\) and (ii) would imply \(\T(b_i) = \T(b_{i+1}) = \cdots = \T(b_j)\), \(\varepsilon = \overline{1}\), and either \(b_j \downarrow b_i\) or \(b_j \rightarrow b_i\). In the former case, semistandardness implies that \(\overline{\T(b_i)} =\overline{1}\), hence by the \(\overline{1}\)-insertion algorithm, \((b_{i+a})_{\overline{0}}= (b_i)_{\overline{0}}+a\) for \(0 \leq a \leq j-i\), so \(b_j\) is in a lower row than \(b_i\), a contradiction.  In the latter case, semistandardness implies that \(\overline{\T(b_i)} =\overline{0}\), hence by the insertion algorithm, \((b_{i+a})_{\overline{1}}= (b_i)_{\overline{1}}+a\) for \(0 \leq a \leq j-i\), so \(b_j\) is in a column rightward of \(b_i\), a contradiction.
\end{proof}

\begin{Lemma}\label{semisemi} Let \(\T\) be a semistandard \((\Xlet,\la)\)-tableau.
\begin{enumerate}
\item  \((\T \xleftarrow{\varepsilon} x)\) is semistandard for every \(x \in\Xlet\).
\item \((\T \xrightarrow{\varepsilon} u)\) is semistandard for every removable node \(u \in [\la]\).
\end{enumerate}
\end{Lemma}
\begin{proof}
(i) Let \(b_1, \ldots, b_m\) be the bumped node sequence, and let \(x_1, \ldots, x_m\) be the bumped letter sequence for the insertion \((\T \xleftarrow{\varepsilon} x)\). For \(1 \leq i \leq m\), let \(\T_i\) be the tableau defined by setting \(\T_i(b_j) = x_j\) for \(1 \leq j \leq i\), and \(\T_i(u) = \T(u)\) for all other \(u \in [\lambda]\). It is easy to see that \(\T_1\) is semistandard. Now we argue that \(\T_k\) is semistandard by induction.

Assume \(\varepsilon + \overline{x}_k = \overline{0}\). Then \((b_k)_{\overline{0}} = (b_{k-1})_{\overline 0} + 1\) by the algorithm. If \(z\) is the node directly below \(b_{k-1}\), then \(z \neq b_i\) for any \(1 \leq i \leq k-1\) by Lemma \ref{inverses}(iv). By semistandardness of \(\T_{k-1}\) then
\begin{align*}
x_k = \T_{k-1}(b_{k-1}) \leq \T_{k-1}(z) = \T(z),
\end{align*}
so \(b_k \rightarrow z\), and thus \(b_k \nearrow b_{k-1}\). Let \(l,r\) be the nodes to the immediate left and right of \(b_k\), respectively. Then 
\begin{align*}
\T_k(l)=\T_{k-1}(l) \leq \T(l) \;\substack{\varepsilon + \overline{1} \\ <} \;x_k \;\substack{\varepsilon \\ <} \;\T(b_k) =\T_{k-1}(b_k)\leq \T_{k-1}(r) = \T_k(r). 
\end{align*}
Thus the \((b_k)_{\overline{0}}\)th row of \(\T_k\) is non-decreasing. Moreover, if \(\overline{x}_k =\overline{1}\), then \(\varepsilon =\overline{1}\), so \(\T_{k}(l) < x_k = \T_k(b_k)\). If \(x_k = \T_k(r) \), then \(\T_{k-1}(b_k) = \T_{k-1}(r)\), yet both are odd, a contradiction of the semistandardness of  \(\T_{k-1}\). Thus \(\T_k(b_k)=x_k < \T_k(r)\). Thus the \((b_k)_{\overline{0}}\)th row of \(\T_k\) is row-strict with respect to odd letters.

Let \(u,d\) be the nodes directly above and below \(b_k\), respectively. Then \(u \rightarrow b_{k-1}\) or \(u=b_{k-1}\). Then 
\begin{align*}
\T_k(u)=\T_{k-1}(u) \leq \T_{k-1}(b_{k-1}) = x_{k-1} \;\substack{\varepsilon \\ <}\;  x_k \;\substack{\varepsilon \\ <}\;\T(b_k) = \T_{k-1}(b_k) \leq \T_{k-1}(d) = \T_k(d).
\end{align*}
Thus the \((b_k)_{\overline{1}}\)th column of \(\T_k\) is non-decreasing. Moreover, if \(\overline{x}_k =\overline{0}\), then \(\varepsilon =\overline{0}\), and \(\T_k(u) < x_k = \T_k(b_k) < \T_k(d)\), so the \((b_k)_{\overline{1}}\)th column of \(\T_k\) is column-strict with respect to even letters.

Thus \(\T_k\) is semistandard if \(\varepsilon + \overline{x}_k= \overline{0}\). Assume on the other hand that \(\varepsilon + \overline{x}_k= \overline{0}\). Let \(\U = \T'^*\). Then, applying the above argument to the insertion \((\U \xleftarrow{  \varepsilon}x^*)\), we have that \(\U_k\) is semistandard. But by Lemma \ref{dualconjinsertion}, \(\U_k =\T_k'^*\). Thus by Lemma \ref{dualconj}, \(\T_k\) is semistandard. This completes the proof of (i).

(ii) Let \(c_1, \ldots, c_m\) be the {\em unbumped node sequence}, and \(y_1, \ldots, y_m\) be the {\em unbumped letter sequence} for the extraction \((\T \xrightarrow{\varepsilon} u)\). Let \(\mu\) be defined such that \([\mu]=[\lambda]\backslash \{u\}\). Let \(\T_1\) be the \((\Xlet,\mu)\)-tableau defined by \(\T_1(v) = \T(v)\) for all \(v \in [\mu]\). For \(2 \leq i \leq m\), let \(\T_i\) be the \((\Xlet,\mu)\)-tableau defined by \(\T_i(c_j) = y_{j-1}\) for \(2 \leq j \leq i\), and \(\T_i(v) = \T(v)\) otherwise. We have that \(\T_1\) is semistandard by assumption. We show by induction that \(\T_k\) is semistandard for all \(k\).

Assume \(\varepsilon +\overline{y}_{k-1} =\overline 0\). Let \(l,r\) be the nodes to the immediate left and right of \(c_k\). Then
\begin{align*}
\T_k(l) = \T_{k-1}(l) \leq \T_{k-1}(c_k) = \T(c_k) \; \substack{\varepsilon \\ <} \; y_{k-1}\; \substack{\varepsilon + \overline{1} \\ <}\; \T(r) = \T_{k}(r).
\end{align*}
Thus the \((c_k)_{\overline 0}\)-th row of \(\T_k\) is non-decreasing. Moreover, if \(\overline{y}_{k-1} = \overline 1\), then \(\varepsilon =  \overline 1\), so \(\T_k(c_k) = y_{k-1} < \T_k(r)\). If \(\T_k(l) = y_{k-1}\), then \(\T_{k-1}(l) = \T_{k-1}(c_k)\), yet both are odd, a contradiction since \(\T_{k-1}\) is semistandard by assumption. Thus \(\T_k\) is row-strict with respect to odd letters.

Let \(u, d\) be the nodes directly above and below \(c_k\), respectively. Then \(c_{k-1} \nearrow c_k\) and \(c_k\) is in the row above \(c_{k-1}\). So \(c_{k-1} \rightarrow d\) or \(c_{k-1} = d\). In either case we have
\begin{align*}
\T_k(u) = \T_{k-1}(u) \leq \T_{k-1}(c_k) = \T_k(c_k) \; \substack{ \varepsilon \\ <} \; y_{k-1} \; \substack{ \varepsilon \\ <} \; y_{k-2} = \T_{k-1}(c_{k-1}) \leq \T_{k-1}(d) = \T_k(d).
\end{align*}
Thus the \((c_k)_{\overline{1}}\)-th column of \(\T_k\) is non-decreasing. Moreover, if \(\overline{y}_{k-1} =\overline{0}\), then \(\varepsilon = \overline 0\), so \(\T_k(u) < y_{k-1} = \T_k(c_k) < \T_k(d)\), and thus \(\T_k\) is column-strict with respect to even letters. 

Therefore \(\T_k\) is semistandard if \(\varepsilon + \overline{y}_{k-1} = \overline{0}\). On the other hand assume \(\varepsilon + \overline{y}_{k-1} = \overline{1}\). Let \(\U = \T'^*\). Then, applying the above argument to the extraction \((\U \xrightarrow{\varepsilon} u')\), we have that \(\U_k\) is semistandard. But \(\U_k = \T_k'^*\) by Lemma \ref{dualconjinsertion}. Thus by Lemma \ref{dualconj}, \(\T_k\) is semistandard.
\end{proof}

\begin{Lemma}
Let \(\T\) be a semistandard \((\Xlet,\lambda)\)-tableau.
\begin{enumerate}
\item \(\T=((\T \xleftarrow{ \varepsilon} x) \xrightarrow{ \varepsilon} A(\T, \varepsilon, x))\) for every \(x \in \Xlet\).
\item \(\T=((\T \xrightarrow{ \varepsilon} u) \xleftarrow{ \varepsilon} R(\T, \varepsilon, u))\) for every removable node \(u \in [\lambda]\).
\end{enumerate}
\end{Lemma}
\begin{proof}
By Lemma \ref{semisemi}, \((\T \xleftarrow{ \varepsilon} x)\) and \((\T \xrightarrow{ \varepsilon} u)\) are semistandard tableaux, and \(\varepsilon\)-insertion and \(\varepsilon\)-extraction are inverse processes by construction.
\end{proof}

 The following lemma is a key tool in generalizing some results proved for standard tableaux to the more general case of semistandard tableaux.

\begin{Lemma}\label{standsame}
Let \(\T\) be a semistandard \((\Xlet, \la)\)-tableau, and let \(\T^\bullet\) be a \(\bullet\)-standardization of \(\T\). 
Let \(y \in \Xlet^\bullet\) be such that
\begin{enumerate}
\item  \(x\prec y\) if \(\varepsilon + \overline{y} = \overline 0\),
\item \(x\succ y\) if \(\varepsilon + \overline{y} = \overline 1\),
\end{enumerate}
for every \(x \in \T^\bullet\) such that \(\hat \bullet(x) =\hat \bullet(y)\). Then \((\T^\bullet\xleftarrow{\varepsilon} y)\) is a \(\bullet\)-standardization of \((\T \xleftarrow{\varepsilon} \hat \bullet(y))\). Moreover, if \(b_1, \ldots, b_k\) is the bumped node sequence for the insertion \((\T \xleftarrow{\varepsilon} \hat \bullet (y))\), and \(b_1^\bullet, \ldots, b_m^\bullet\) is the bumped node sequence for the insertion \((\T^\bullet \xleftarrow{\varepsilon} y)\), then \(k=m\) and \(b_i = b_i^\bullet\) for all \(i\).
\end{Lemma}

\begin{proof}
We will first prove the result in the case \(\varepsilon = \overline 0\), so that \(x < y\) for all \(x \in \T^\bullet\) such that \(\hat \bullet(x) = \hat \bullet (y)\). We will write \(\T_y^\bullet\) for \((\T^\bullet \xleftarrow{\overline 0} y)\) and \(\T_{\hat\bullet(y)}\) for \((\T \xleftarrow{\overline{0}} \hat\bullet(y))\). We will also write \(x_1,\ldots, x_k\) be the bumped letter sequence for the insertion \((\T \xleftarrow{\varepsilon} \hat \bullet (y))\), and \(x_1^\bullet, \ldots, x_m^\bullet\) for the bumped letter sequence for the insertion \((\T^\bullet \xleftarrow{\varepsilon} y)\). We will first prove by induction that \(b_i = b_i^\bullet\) for all \(i\), hence \(k=m\).

We have \(y < \T^\bullet(b_1^\bullet)\), so \(\hat\bullet(y) \leq \hat\bullet(\T^{\bullet}(b_1^\bullet)) = \T(b_1^\bullet)\). Moreover, by the assumption on \(y\), we have that \(\hat\bullet(y) \neq \T(b_1^\bullet)\), so \(\hat\bullet(y) < \T(b_1^\bullet)\). If \(\overline{y}=\overline 0\) (resp. if \(\overline{y}=\overline 1\)), let \(u\) be the node directly to the left (resp. directly above) of \(b_1^\bullet\). Then \(\T^\bullet(u) < y\), so \(\T(u) \leq \hat\bullet(y)\). Thus \(b_1 =b_1^\bullet\).

Now assume that \(b_i = b_i^\bullet\). If \(\overline{\T^\bullet(b_i^\bullet)} = \overline{0}\) (resp. if \(\overline{\T^\bullet(b_i^\bullet)} = \overline{1}\)), then \(b_{i+1}^\bullet \nearrow b_i\) (resp. \(b_i \nearrow  b_{i+1}^\bullet\)), so \(\T^\bullet(b_{i+1}^\bullet) < \T^\bullet(b_i)\) if \(\T(b_{i+1}^\bullet) = \T(b_i)\), since \(\T^\bullet\) is a \(\bullet\)-standardization of \(\T\). However, we also have \(\T^\bullet(b_i) < \T^\bullet(b^\bullet_{i+1})\), so \(\T(b_i) \leq \T(b_{i+1}^\bullet)\), and hence \(x_{i+1} = \T(b_i) < \T(b_{i+1}^\bullet)\). If \(\overline{\T^\bullet(b_i^\bullet)} =\overline 0\) (resp. if \(\overline{\T^\bullet(b_i^\bullet)}=\overline{1}\)), let \(u\) be the node directly to the left (resp. directly above) \(b_{i+1}^\bullet\). Then \(\T^\bullet(u) < \T^\bullet(b_i)\), so \(\T(u) \leq \T(b_i) = x_{i+1}\), and thus \(b_{i+1}=b_{i+1}^\bullet\).

Therefore we have that \(\hat\bullet( \T^\bullet_y) = \T_{\hat\bullet(y)}\), and by construction \(\T^\bullet_y\) is standard. Now we show that \(\T^\bullet_y\) is \(\bullet\)-standard. Let \(u \nearrow v \in [\textup{sh}(\T^\bullet_y)]\), and assume \(\T_{\hat\bullet(y)}(u)= \T_{\hat\bullet(y)}(v)\). If neither \(u\) nor \(v\) is equal to a bumped node \(b_i^\bullet\), then the result follows since \(\T^\bullet\) is \(\bullet\)-standard. 
On the other hand if both are bumped nodes, say \(u=b_i\) and \(v = b_j\), then \(\T^\bullet_y(u) = x_i^\bullet\), \(\T^\bullet_y(v) = x_j^\bullet\), and \(x_i = x_j\). But this cannot happen in \(\bar 0\)-insertion for distinct \(i\) and \(j\). This leaves the cases where exactly one of \(u\), \(v\) is a bumped node. We consider the two cases separately:

\begin{enumerate}
\item[(a)]
Assume that \(v=b_i\) is a bumped node, and \(u\) is not. Then \(\T_y^\bullet(v) = x_i^\bullet\), and \(\T(u) = \T_{\hat\bullet(y)}(u)=\T_{\hat\bullet(y)}(v) = x_i\).
Note that if \(\overline{y}=\overline{1}\), then \(\overline{\hat\bullet(y)}=\overline{1}\), and \(b_1\) is in the first column. Then by Lemma \ref{inverses}, \(\hat\bullet(y)=\T_{\hat\bullet(y)}(b_1) < \T_{\hat\bullet(y)}(w)\) for every node \(w\) directly below \(b_1\). Then, since \(\T_{\hat\bullet(y)}\) is semistandard, \(\T_{\hat\bullet(y)}(b_1) < \T_{\hat\bullet(y)}(w)\) for every node \(w\) such that \(b_1 \searrow w\). Thus \(b_1 \nearrow w\) for every node \(w\) such that \(\hat\bullet(y)=\T_{\hat\bullet(y)}(b_1)=\T_{\hat\bullet(y)}(w)\). Thus if \(i=1\), then \(\overline{y}=\overline 0\) and by the assumption on \(y\),  \(\T^\bullet_y(u)=\T^\bullet(u) < y = x_1^\bullet = \T_y^\bullet(v)\) as required.  

Assume \(i\geq 2\). Then \(x_i^\bullet = \T^\bullet(b_{i-1})\). If \(\overline{x}_i = \overline 0\), then \(b_i \nearrow b_{i-1}\). Thus \(u \nearrow b_{i-1}\), so \(\T^\bullet_y(u) = \T^\bullet(u) < \T^\bullet(b_{i-1}) = \T_y^\bullet (v)\) since \(\T^\bullet\) is \(\hat\bullet\)-standard. Assume \(\overline{x}_i =\overline 1\). Then \( b_{i-1} \nearrow b_{i}=v\). If \(v\) is directly above \(u\), then \(\T_y^\bullet(v) < \T_y^\bullet(u)\) since \(\T_y^\bullet\) is standard.
Assume \(v\) is not directly above \(u\). If it is not the case that \(u \nearrow b_{i-1}\), then it must be that \(u \searrow b_{i-1}\). But then since \(\T_y(b_{i-1}) = x_{i-1} < \T(b_{i-1})=\T(u)=\T_y(u)\), this cannot be true. Therefore \(u \nearrow b_{i-1}\), and again  we have \(\T^\bullet_y(u) = \T^\bullet(u) > \T^\bullet(b_{i-1}) = \T_y(v)\), as required.

\item[(b)]
Assume that \(u = b_i\) is a bumped node, and \(v\) is not. Then \(\T_y^\bullet(u) = x_i^\bullet\), and \(\T(v) = \T_{\hat\bullet(y)}(v) = \T_{\hat\bullet(y)}(u)=x_i\). Note that if \(\overline{y} = \overline{0}\), then \(w \nearrow b_1\) for every node \(w\) such that \(\hat\bullet(y)=\T_{\hat\bullet(y)}(b_1)=\T_{\hat\bullet(y)}(w)\). Thus if \(i=1\), then \(\overline{y}=\overline 1\) and \(\T^\bullet(u) = y = x_1^\bullet > \T^\bullet(v)\) as required.

Assume \(i\geq 2\). Then \(x_i^\bullet = \T^\bullet(b_{i-1})\). If \(\overline{x}_i=\overline{1}\), then \(b_{i-1} \nearrow b_{i}\). Thus \(b_{i-1} \nearrow v\), so \(\T^\bullet_y(u) = \T^\bullet(b_{i-1}) > \T^\bullet(v) = \T_y^\bullet(v)\). Assume \(\overline{x}_i = \overline 0\). If \(v\) is directly to the right of \(u\), then \(\T_y^\bullet(u) <\T_y^\bullet(v)\). Assume \(v\) is not directly to the right of \(u\). Then if it is not the case that \(b_{i-1} \nearrow v\), then it must be that \(v \searrow b_{i-1}\). But then since \(\T_y(b_{i-1}) = x_{i-1} < \T(b_{i-1})=\T(v)=\T_y(v)\), this cannot be true. Therefore \(b_{i-1} \nearrow v\), and thus we have \(\T^\bullet_y(u) = \T^\bullet(b_{i-1}) > \T^\bullet(v) = \T_y(v)\), as required. 
\end{enumerate}

This completes the proof of the lemma when \(\varepsilon = \overline{0}\). Now assume \(\varepsilon = \overline 1\). Then \(x > y\) for all \(x \in \T^\bullet\). This proof proceeds along the same lines as the first part, but because there is an inherent discrepancy in the comparisons \(\substack{\overline 0 \\ <} = <\) and \(\substack{ \overline 1 \\ <} = \leq\)  we will provide the details in full. We'll write \(T_y^\bullet\) for \((T^\bullet \xleftarrow{\overline 1} y)\) and \(\T_{\hat\bullet(y)}\) for \((\T \xleftarrow{\overline 1} \hat\bullet(y))\). We will prove by induction that \(b_i = b_i^\bullet\) for all \(i\), hence \(k=m\).

We have \(y < \T^\bullet(b_1^\bullet)\), so \(\hat\bullet(y) \leq \hat\bullet(\T^\bullet(b_1^\bullet)) = \T(b_1^\bullet)\) since \(\T^\bullet\) is a \(\bullet\)-standardization of \(\T\).  If \(\overline{y} =\overline 0\) (resp. if \(\overline y = \overline 1\)), let \(u\) be the node directly above (resp. directly to the left of) \(b_1^\bullet\). Then \(\T^\bullet(u) < y\), so \(\T(u) \leq \hat\bullet(y)\). Moreover, by the assumption on \(y\), we have that \(\hat\bullet(y) \neq \T(u)\), so \(\hat\bullet(y) > \T(u)\). Thus \(b_1 =b_1^\bullet\).

Now assume that \(b_i = b_i^\bullet\). We have \(x^\bullet_{i+1} = \T^\bullet(b_i) < \T^\bullet(b_{i+1}^\bullet)\), so \(x_{i+1}=\T(b_i) \leq \T(b_{i+1}^\bullet)\). If \(\overline{\T^\bullet(b_i)}=\overline 1\) (resp. \(\overline{\T^\bullet(b_i)} = \overline 0\)), let \(u\) be the node directly to the left of (resp. directly above) \(b_{i+1}^\bullet\), and note that \(u\nearrow b_i\) (resp. \(b_i \nearrow u\)), so that \(\T^\bullet(u)> \T^\bullet(b_i) = x^\bullet_{i+1}\) if \(\T(u) = \T(b_i) = x_{i+1}\). But \(x_{i+1}^\bullet > \T^\bullet(u)\) by the definition of \(b_{i+1}^\bullet\), so it must be that \(\T(u) \neq x_{i+1}\). Therefore \(x_{i+1} > \T(u)\), and so \(b_{i+1} = b^\bullet_{i+1}\).

Therefore we have that \(\hat\bullet( T^\bullet_y) = T_{\hat\bullet(y)}\), and by construction \(\T^\bullet_y\) is standard. Now we show that \(\T^\bullet_y\) is \(\bullet\)-standard. Let \(u \nearrow v \in [\textup{sh}(\T^\bullet_y)]\), and assume \(\T_{\hat\bullet(y)}(u)= \T_{\hat\bullet(y)}(v)\). If neither \(u\) nor \(v\) is equal to a bumped node \(b_i^\bullet\), then the result follows since \(\T^\bullet\) is \(\bullet\)-standard. 

On the other hand if both are bumped nodes, say \(u=b_i\) and \(v = b_j\), then \(\T^\bullet_y(u) = x_i^\bullet\), \(\T^\bullet_y(v) = x_j^\bullet\), and \(x_i = x_j\). Note in general that if \(x_k = x_{k+1}\), then \(b_k \nearrow b_{k+1}\) if \(\overline{x}_k =\overline 0\), and \(b_{k+1} \nearrow b_k\) if \(\overline{x}_k =\overline 1\). Thus, if \(i<j\), then \(x_i = x_{i+1} = \cdots = x_j\) and \(b_i \nearrow b_j\) imply that \(\overline{x}_i =\overline 0\) and \(\T^\bullet_y(u) = \T_y^\bullet(b_i) = x_i^\bullet  < x_j^\bullet = \T_y^\bullet(b_j) = \T^\bullet_y(v)\), as required. On the other hand, if \(j<i\), then \(x_j = x_{j+1} = \cdots = x_i\) and \(b_i \nearrow b_j\) imply that \(\overline{x}_i = \overline 1\) and \(\T^\bullet_y(u) = \T_y^\bullet(b_i) = x_i^\bullet  > x_j^\bullet = \T_y^\bullet(b_j) = \T^\bullet_y(v)\), as required. This leaves the cases where exactly one of \(u,v\) is a bumped node. We consider the two cases separately:

\begin{enumerate}
\item[(a)]
Assume that \(v=b_i\) is a bumped node, and \(u\) is not. Then \(\T_y^\bullet(v) = x_i^\bullet\), and \(\T_{\hat\bullet(y)}(u) = \T_{\hat\bullet(y)}(v) = x_i\). Note that if \(\overline{y}= \overline 0\), then \(b_1 \nearrow w\) for every node \(w\) such that \(\hat\bullet(y)=\T_{\hat\bullet(y)}(w)\). Thus if \(i=1\), then \(\overline{y}=\overline 1\) and \(\T_y^\bullet(u) =\T^\bullet(u) > y = x_1^\bullet = \T_y^\bullet(v)\) as required. 

Assume \(i\geq 2\). Then \(x_i^\bullet = \T^\bullet(b_{i-1})\). If \(\overline{x}_i=\overline{1}\), then \(b_i \nearrow b_{i-1}\). Thus \(u \nearrow b_{i-1}\), so \(\T^\bullet_y(u) = \T^\bullet(u) > \T^\bullet(b_{i-1}) = \T_y^\bullet (v)\). Assume \(\overline{x}_i =\overline 0\). Since \(\T_{\hat\bullet(y)}(u) = \T_{\hat\bullet(y)}(v)=x_i\) and \(\T_{\hat\bullet(y)}\) is semistandard, it cannot be that \(v\) is directly above \(u\). If it is not the case that \(u \nearrow b_{i-1}\), then it must be that \(u \searrow b_{i-1}\). Moreover since \(\T(u) = \T(b_{i-1})\) and \(\T\) is semistandard, it cannot be that \(u\) is directly above \(b_{i-1}\), so \(u \SEarrow b_{i-1}\). But then \(\T_{\hat\bullet(y)}(u) = \T(u) < \T(b_{i-1}) = \T_{\hat\bullet(y)}(v)\), a contradiction. Therefore \(u \nearrow b_{i-1}\), and thus we have \(\T^\bullet_y(u) = \T^\bullet(u) < \T^\bullet(b_{i-1}) = \T^\bullet_y(v)\), as required.

\item[(b)]
Assume that \(u = b_i\) is a bumped node, and \(v\) is not. Then \(\T^\bullet(u) = x_i^\bullet\), and \(\T_{\hat\bullet(y)}(u) = \T_{\hat\bullet(y)}(v) =\T(v)= x_i\). Note that if \(\overline{y}=\overline{1}\), then \(w \nearrow b_1\) for every node \(w\) such that \(\hat\bullet(y)=\T_{\hat\bullet(y)}(w)\). Thus if \(i=1\), then \(\overline{y}=\overline 0\) and \(\T^\bullet(u) = y = x_1^\bullet < \T^\bullet(v)\) as required.  

Assume \(i\geq 2\). Then \(x_i^\bullet = \T^\bullet(b_{i-1})\). If \(\overline{x}_i=\overline{0}\), then \(b_{i-1} \nearrow b_{i}\). Thus \(b_{i-1} \nearrow v\), so \(\T^\bullet_y(u) = \T^\bullet(b_{i-1}) < \T^\bullet(v) = \T_y^\bullet(v)\).   Assume \(\overline{x}_i =\overline{1}\). Since \(\T_{\hat\bullet(y)}(u) = \T_{\hat\bullet(y)}(v) =x_i\) and \(\T_{\hat\bullet(y)}\) is semistandard, it cannot be that \(v\) is directly to the right of \(u\). If it is not the case that \(b_{i-1} \nearrow v\), then it must be that \(v \searrow b_{i-1}\). Moreover since \(\T(v) = \T(b_{i-1})\) and \(\T\) is semistandard, it cannot be that \(b_{i-1}\) is directly to the right of \(v\), so \(v \SEarrow b_{i-1}\). But then \(\T_{\hat\bullet(y)}(u) = \T(b_{i-1}) > \T(v)= \T_{\hat\bullet(y)}(v)\), a contradiction. Therefore \(b_{i-1} \nearrow v\), and thus we have \(\T^\bullet_y(u) = \T^\bullet(b_{i-1})  > \T^\bullet(v) = \T^\bullet_y(v)\), as required.
\end{enumerate}
This completes the proof of the lemma in the case \(\varepsilon = \overline{1}\).
\end{proof}

\section{Behavior of bumped nodes}\label{behavsec}

\subsection{Bumped node distribution} 
In this section we prove some technical results on the distribution of bumped nodes in \(\varepsilon\)-insertion, which will be of repeated use in \S\ref{successive}.

\begin{Lemma}\label{filler}
Let \(\T\) be a semistandard \((\Xlet, \lambda)\)-tableau, \(\varepsilon,\delta  \in \Z_2\), and \(x \in \Xlet\). Let \(b_1, \ldots, b_m\) be the bumped node sequence for the insertion \((\T \xleftarrow{ \varepsilon} x)\). If \(i<j\) and \((b_i)_\delta < (b_j)_\delta\), then there exists a sequence \(i \leq t_0 < \cdots < t_k < j\), where \(k = (b_j)_\delta - (b_i)_\delta -1\), such that \((b_{t_a})_\delta = (b_i)_\delta + a\), and \(\varepsilon + \overline{\T(b_{t_a})} = \delta\) for all \(a\).
\end{Lemma}
\begin{proof}
Let \(l\) be minimal such that \((b_{i+l})_\delta \geq (b_{i})_\delta+1\). Then \(i<l \leq j\). If \(\varepsilon + \overline{\T(b_{i+l-1})} = \delta + 1\), then the algorithm implies that \((b_{i+l})_\delta \leq (b_{i+l -1})_\delta\), a contradiction of the minimality of \(l\). Thus \(\varepsilon + \overline{\T(b_{i+l-1})} = \delta \), hence \((b_{i+l})_\delta = (b_{i+l -1})_\delta +1\), so by minimality of \(l\), we must have \((b_{i+l-1})_\delta = (b_i)_\delta\). Set \(t_0 = i+l-1\). Then \(i\leq t_0 < t_0+1 \leq j\), and \((b_{t_0 +1})_\delta = (b_i)_\delta+1\). Now the claim follows by induction.
\end{proof}

\begin{Lemma}\label{triple}
Let \(\T\) be a semistandard \((\Xlet, \lambda)\)-tableau, \(\varepsilon \in \Z_2\), and \(x \in \Xlet\). Let \(b_1, \ldots, b_m\) be the bumped node sequence for the insertion \((\T \xleftarrow{\varepsilon} x)\). Let \(i,j,k\) be such that
\begin{enumerate}
\item \(i,j < k\), 
\item \(b_i \NEarrow b_k \NEarrow b_j\),
\item \(((b_i)_{\overline{0}}, (b_j)_{\overline{1}}-1), ((b_i)_{\overline{0}}-1, (b_j)_{\overline{1}}) \in [\textup{sh}(\T \xleftarrow{ \varepsilon} x)]\).
\end{enumerate}
Then there exists some \(l>k\) such that
\begin{enumerate}
\item \(b_i \Rightarrow b_l\) and \(b_l \NEarrow b_i\), or;
\item \(b_j \Downarrow b_l\) and \(b_i \NEarrow b_l\).
\end{enumerate}
\end{Lemma}

\begin{proof}

We will call a triple \((i,k,j)\) which satisfies (i)--(iii) a {\em stair triple}. For compactness we'll write \(r_a\) for \((b_a)_{\overline 0}\) and \(c_a\) for \((b_a)_{\overline{1}}\). Define:
\begin{align*}
m_{i,j,k}&:=(c_j - c_k)+(r_i - r_k)\\
 n_{i,j}&:=(c_j-c_i) + (r_i-r_j)
\end{align*}
Note that \(2 \leq m_{i,j,k} \leq n_{i,j} -2\).

Take \({n}_{i,j}=4\), the least possible value for \(n_{i,j}\). Then \(m_{i,j,k} = 2\). Then \(b_k = (r_i-1,c_j-1)\), so \(b_k \in [\la]\) and thus cannot be the last bumped node. By Lemma \ref{inverses}, either \(b_{k+1} = (r_i, c_{i+1})\) (if \(\varepsilon + \overline{\T(b_i)}= \overline 0\)) or \(b_{k+1} = (r_j+1, c_{j})\) (if \(\varepsilon + \overline{\T(b_i)}= \overline 1\)). Taking \(l = k+1\), this completes the base case.

We argue by induction. Assume that \(i,j,k\) satisfy (i)-(iii), and further assume that the claim holds for all \(i',j',k'\) such that \(n_{i',j'} < n_{i,j}\), or \(n_{i',j'}=n_{i,j}\) and \(m_{i',j',k'} < m_{i,j,k}\).

Assume \(\varepsilon +\overline{\T(b_k)}=\overline{0}\) (the argument in the other case is exactly dual to what follows). Then \(r_{k+1} = r_k + 1\). If \(r_{k+1} = r_j\), then, taking \(l=k+1\), we are in case (i). Assume \(r_{k+1}<r_i\). If \(c_{k+1} = c_k\), we may apply the induction assumption to the stair triple \((i, k+1, j)\). Thus assume \(c_{k+1} < c_k\). Since \(k+1 > j\), it must be that \(c_{k+1} > c_j\). Now, apply the induction assumption to the stair triple \((i, k+1, k)\). This either gives a node \(b_l\) which satisfies (i), or \(b_l\) is such that \(c_l = c_k\) and \(r_k < r_l < r_i\). In the former case we are done, so assume the latter. Now apply the induction assumption to the stair triple \((i, l, j)\), and we are done.
\end{proof}

\subsection{Bumped nodes in successive insertions}\label{successive}

In this section we prove a key result which bounds the distribution of bumped nodes appearing in successive insertions. Theorem \ref{nodecontain}, together with Corollary \ref{addednodes} can be viewed as a generalization of \cite[Theorem 1]{Knuth} to the realm of superalphabets and \(\varepsilon\)-insertion. We begin by defining a certain set partition of the nodes of a Young diagram that naturally results from the \(\varepsilon\)-insertion process.

Let \(y \in \Xlet\), \(\varepsilon \in \Z_2\), and let \(\T\) be a semistandard \((\Xlet, \la)\)-tableau. Let \(b_1^y, \ldots, b_{m}^y\) be the bumped nodes for the insertion \((\T \xleftarrow{ \varepsilon} y)\). Assume \(\textup{sh}(\T \xleftarrow{\varepsilon} y) = \mu\). Let \(r_i^y = (b_i^y)_{\overline{0}}\) and \(c_i^y=(b_i^y)_{\overline{1}}\) for and \(1 \leq i \leq m\). For arbitrary \(v \in \Nodes\) we set 
\begin{align*}
l(v) &= \max\left[ \{ j \mid r_j^y = v_{\overline{0}},\; c_j^y < v_{\overline{1}}\} \cup \{0\} \right]\\
u(v) &= \max\left[ \{ j \mid c_j^y =v_{\overline{1}}, \;r_j^y < v_{\overline{0}}\} \cup \{0\} \right].
\end{align*}

Then \(b^y_{l(v)}\) is the nearest node directly to the left of \(v\) which was bumped in the \(y\) insertion. If no such element exists then \(l(v) = 0\).  Similarly, \(b^y_{u(v)}\) is the nearest node directly above \(v\) which was bumped in the \(y\) insertion. If no such element exists then \(u(v) = 0\). Let \([\mu]_A \) be the set of all nodes of \([\mu]\) together with all addable nodes of \([\mu]\).

Now define the sets 
\begin{align*}
\textup{NE}(\T, \varepsilon, y)&:=
\{v \in [\mu]_A \mid l(v) > u(v)\} \cup \{v \in [\mu]_A \mid l(v) = u(v) = 0, \varepsilon + \overline{y}= \overline 1\}\\
\textup{SW}(\T, \varepsilon, y)&:=
\{v \in [\mu]_A \mid l(v) < u(v)\} \cup \{v \in [\mu]_A \mid l(v) = u(v) = 0, \varepsilon + \overline{y}= \overline 0\}
\end{align*}

\begin{Remark}
Informally, \(\textup{NE}(\T, \varepsilon, y)\) represents the set of nodes `northeast' of a rough perimeter delineated by tracing the path of the bumped nodes in sequence, and \(\textup{SW}(\T, \varepsilon, y)\) represents the nodes to the `southwest' of that perimeter. For example, if \(\varepsilon + \overline{y} = \overline 1\), and the bumped nodes are those labeled in the diagram of \([\mu]\) below, then the red-colored nodes represent the set \(\textup{NE}(\T, \varepsilon, y)\), and the blue-colored nodes represent the set \(\textup{SW}(\T, \varepsilon, y)\).

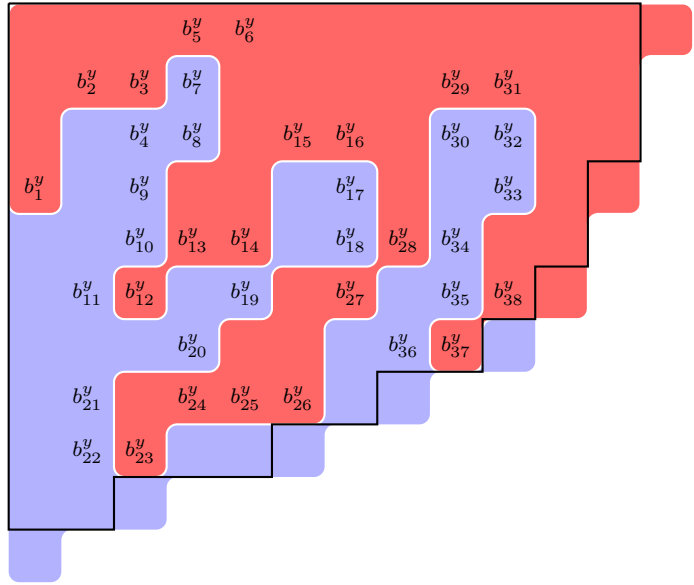
\begin{align*}
\begin{tikzpicture}[scale=0.7]
\tikzset{baseline=-9mm}
\fill [ rounded corners=4pt, fill=blue!30] (0,-1)--(0,-12)--(1,-12)--(1,-11)--(3,-11)--(3,-10)--(6,-10)--(6,-9)--(8,-9)--(8,-8)--(10,-8)--(10,-6)--(11,-7)--(11,-1)--cycle;
\fill [ rounded corners=4pt, fill=red!60] (0,-1)--(13,-1)--(13,-2)--(12,-2)--(12,-5)--(11,-5)--(11,-7)--(9,-7)--(9,-8)--(8,-8)--(8,-7)--(9,-7)--(9,-5)--(10,-5)--(10,-3)--(8,-3)--(8,-6)--(7,-6)--(7,-7)--(6,-7)--(6,-9)--(3,-9)--(3,-10)--(2,-10)--(2,-8)--(3,-8)--(4,-8)--(4,-7)--(5,-7)--(5,-6)--(7,-6)--(7,-4)--(5,-4)--(5,-6)--(3,-6)--(3,-7)--(2,-7)--(2,-6)--(3,-6)--(3,-4)--(4,-4)--(4,-2)--(3,-2)--(3,-3)--(1,-3)--(1,-5)--(0,-5)--cycle;
\draw [ white, line width=0.8pt, rounded corners=4pt, fill=none] (0,-1)--(13,-1)--(13,-2)--(12,-2)--(12,-5)--(11,-5)--(11,-7)--(9,-7)--(9,-8)--(8,-8)--(8,-7)--(9,-7)--(9,-5)--(10,-5)--(10,-3)--(8,-3)--(8,-6)--(7,-6)--(7,-7)--(6,-7)--(6,-9)--(3,-9)--(3,-10)--(2,-10)--(2,-8)--(3,-8)--(4,-8)--(4,-7)--(5,-7)--(5,-6)--(7,-6)--(7,-4)--(5,-4)--(5,-6)--(3,-6)--(3,-7)--(2,-7)--(2,-6)--(3,-6)--(3,-4)--(4,-4)--(4,-2)--(3,-2)--(3,-3)--(1,-3)--(1,-5)--(0,-5)--cycle;
\draw[ thick ](0,-1)--(12,-1)--(12,-4)--(11,-4)--(11,-6)--(10,-6)--(10,-7)--(9,-7)--(9,-8)--(7,-8)--(7,-9)--(5,-9)--(5,-10)--(2,-10)--(2,-11)--(0,-11)--(0,-1);
\draw(0.5,-4.5) node{$\scriptstyle b_1^y$};
\draw(1.5,-2.5) node{$\scriptstyle b_2^y$};
\draw(2.5,-2.5) node{$\scriptstyle b_3^y$};
\draw(2.5,-3.5) node{$\scriptstyle b_4^y$};
\draw(3.5,-1.5) node{$\scriptstyle b_5^y$};
\draw(4.5,-1.5) node{$\scriptstyle b_6^y$};
\draw(3.5,-2.5) node{$\scriptstyle b_7^y$};
\draw(3.5,-3.5) node{$\scriptstyle b_8^y$};
\draw(2.5,-4.5) node{$\scriptstyle b_9^y$};
\draw(2.5,-5.5) node{$\scriptstyle b_{10}^y$};
\draw(1.5,-6.5) node{$\scriptstyle b_{11}^y$};
\draw(2.5,-6.5) node{$\scriptstyle b_{12}^y$};
\draw(3.5,-5.5) node{$\scriptstyle b_{13}^y$};
\draw(4.5,-5.5) node{$\scriptstyle b_{14}^y$};
\draw(5.5,-3.5) node{$\scriptstyle b_{15}^y$};
\draw(6.5,-3.5) node{$\scriptstyle b_{16}^y$};
\draw(6.5,-4.5) node{$\scriptstyle b_{17}^y$};
\draw(6.5,-5.5) node{$\scriptstyle b_{18}^y$};
\draw(4.5,-6.5) node{$\scriptstyle b_{19}^y$};
\draw(3.5,-7.5) node{$\scriptstyle b_{20}^y$};
\draw(1.5,-8.5) node{$\scriptstyle b_{21}^y$};
\draw(1.5,-9.5) node{$\scriptstyle b_{22}^y$};
\draw(2.5,-9.5) node{$\scriptstyle b_{23}^y$};
\draw(3.5,-8.5) node{$\scriptstyle b_{24}^y$};
\draw(4.5,-8.5) node{$\scriptstyle b_{25}^y$};
\draw(5.5,-8.5) node{$\scriptstyle b_{26}^y$};
\draw(6.5,-6.5) node{$\scriptstyle b_{27}^y$};
\draw(7.5,-5.5) node{$\scriptstyle b_{28}^y$};
\draw(8.5,-2.5) node{$\scriptstyle b_{29}^y$};
\draw(8.5,-3.5) node{$\scriptstyle b_{30}^y$};
\draw(9.5,-2.5) node{$\scriptstyle b_{31}^y$};
\draw(9.5,-3.5) node{$\scriptstyle b_{32}^y$};
\draw(9.5,-4.5) node{$\scriptstyle b_{33}^y$};
\draw(8.5,-5.5) node{$\scriptstyle b_{34}^y$};
\draw(8.5,-6.5) node{$\scriptstyle b_{35}^y$};
\draw(7.5,-7.5) node{$\scriptstyle b_{36}^y$};
\draw(8.5,-7.5) node{$\scriptstyle b_{37}^y$};
\draw(9.5,-6.5) node{$\scriptstyle b_{38}^y$};
\end{tikzpicture}
\\
\end{align*}
Though we will not need this fact, it follows from the definition that \(b_i^y \in \textup{NE}(\T, \varepsilon, y)\) if and only if \(\overline{(\T\xleftarrow{\varepsilon} y)(b^y_{i})} + \varepsilon=\overline 1\), as can be verified in the example above.
\end{Remark}

\begin{Theorem}\label{nodecontain}
Assume \(\varepsilon \in \ZZ_2\), \(y,z  \in \Xlet\), and  \(\T\) is a semistandard \((\Xlet,\lambda)\)-tableau. Let \(b_1^y, \ldots, b_{m_1}^y\) and \(b^z_1, \ldots, b^z_{m_2}\) be the bumped node sequences for the insertions \(\T_y:=(\T \xleftarrow{\varepsilon} y)\) and \(\T_z:=(\T_y \xleftarrow{\varepsilon} z)\) respectively. Then 
\begin{align*}
\{b^z_1, \ldots, b_{m_2}^z\} \subseteq \textup{NE}(\T, \varepsilon, y) 
\hspace{3mm}
\iff
\hspace{3mm}
\begin{cases}
y\prec z \textup{ and } \varepsilon = \overline{0}, \textup{ or}\\
y\succ  z \textup{ and } \varepsilon = \overline{1}, \textup{ or}\\
y = z \textup{ and } \overline{y} = \overline{0},
\end{cases}
\end{align*}
and
\begin{align*}
\{b^z_1, \ldots, b_{m_2}^z\} \subseteq \textup{SW}(\T, \varepsilon, y) 
\hspace{3mm}
\iff
\hspace{3mm}
\begin{cases}
y\succ z \textup{ and } \varepsilon = \overline{0}, \textup{ or}\\
y\prec  z \textup{ and } \varepsilon = \overline{1}, \textup{ or}\\
y = z \textup{ and } \overline{y} = \overline{1}.
\end{cases}
\end{align*}
\end{Theorem}

\begin{proof} 
Since \(\textup{NW}(\T, \varepsilon, y) \sqcup \textup{SE}(\T, \varepsilon,y) = [\mu]_A\), we may prove the equivalent statement:
\begin{align}\label{nodecontaineq}
b_i^z \in \textup{NE}(\T, \varepsilon, y) 
\hspace{3mm}
\iff
\hspace{3mm}
\begin{cases}
y\prec z \textup{ and } \varepsilon = \overline{0}, \textup{ or}\\
y\succ  z \textup{ and } \varepsilon = \overline{1}, \textup{ or}\\
y = z \textup{ and } \overline{y} = \overline{0},
\end{cases}
\end{align}
for all \(1 \leq i \leq m_2\). First we prove that the lemma holds when \(\T\) is a {\em standard} tableau, \(y,z \notin \T\), \(y \neq z\), and \(\overline y = \overline 0\). Note that in this situation we need not consider the third case in the right side of (\ref{nodecontaineq}). We will go by induction on \(1 \leq i \leq m_2\). We will write \(\textup{NE}\) for \(\textup{NE}(\T, \varepsilon, y)\) where the context is clear.

\underline{Base case \(i=1\).}
Assume \(y \prec z\). Since \(\overline{y} = \overline{0}\), we have \(y < z\) and \(\overline{z} = \overline{0}\). If \(\varepsilon = \overline{0}\), then \(b_1^y\) is in the first row, and \(\T_y(b_1^y)=y < z\), so \(b_1^y \Rightarrow b_1^z\). Then \(l_1 >0 = u_1\), so \(b_1^z \in \textup{NE}\). If \(\varepsilon = \overline{1}\), then \(b_1^y\) is in the first column, and \(b_1^z \UParrow b_1^y\) so \(l_1 = 0 < u_1\), so \(b_i^z \notin \textup{NE}\).

Now assume \(y\succ   z\) and \(\varepsilon = \overline{0}\). Then \(b_1^y\) is in the first row. If \(\overline{z} = \overline{0}\), then \(z < y\), so \(b_1^z \Rightarrow b_1^y\). Then \(l_1 = 0\), so \(b_1^z \notin \textup{NE}\). If \(\overline{z} = \overline{1}\), then \(b_1^z\) is in the first column, so \(l_1 = 0\), and again \(b_1^z \notin \textup{NE}\). 

Now assume \(y\succ  z\) and \(\varepsilon = \overline{1}\). Then \(b_1^y\) is in the first column. If \(\overline{z} = \overline{0}\), then \(z < y\), so \(b_1^y \UParrow b_1^z\). Then \(u_1 = 0\), so \(b_1^z \in \textup{NE}\). If \(\overline{z} = \overline{1}\), then \(b_1^z\) is in the first row, so again \(u_1 = 0\) and \(b_1^z \in \textup{NE}\).

\underline{Induction step.} So (\ref{nodecontaineq}) holds when \(i=1\). Now we show that
\begin{align*}
b_i^z \in \textup{NE} \hspace{5mm} \iff \hspace{5mm} b_{i+1}^z \in \textup{NE}.
\end{align*}
\(( \implies)\) Assume \(b_i^z \in \textup{NE}\). Then \(l_i > u_i\), or \(l_i = u_i = 0\) and \(\varepsilon=\overline 1\). We assume the former, and will later address the latter case. If \(b_k^y\neq b_{u_i}^y\) is a node such that \(b_k^y \searrow b_{u_i}^y\), then \(k<u_i\). By Lemma \ref{triple}, if \(b_k^y \) is a node such that \(b^y_{l_i}\NEarrow b_k^y \NEarrow b_{u_i}^y\), then \(k<l_i\). From this it follows that \(\varepsilon + \T(b_{l_i}^y) = \overline{0}\), so \((b^y_{l_i + 1})_{\overline{0}} = (b^y_{l_i})_{\overline{0}} + 1\). There are two cases to consider:

\begin{enumerate}
\item[(a)] Assume \(\varepsilon  + \overline{\T_y(b_i^z)} = \overline 0\). Then \((b_{i+1}^z)_{\overline 0} = (b_i^z)_{\overline 0}+1\). First we show that \(b_i^z \neq b_r^y\) for any \(r\). Indeed, if we did have that \(b_i^z = b^y_r\), then by Lemma \ref{filler} applied to \((b_{l_i}^y, b_{r}^y)\), there is some \(b_s^y\) in the column to the left of \(b_r^y\), with \(l_i \leq s < i\) and \(\varepsilon + \overline{ \T(b_s^y)} =\overline{1}\). Then by the above paragraph, \(b^y_s \nearrow b^y_i\). But then \(b_{s+1}^y\) is in the same column as \(b_i^y\), but cannot be above or below \(b_i^y\) since \(u_i<s< s+1 \leq i\). Then the only option is \(s+1 = r\). But this cannot be, since by assumption 
\begin{align*}
\varepsilon +\overline{\T(b_{s}^y)} = \varepsilon +\overline{\T(b_{r-1}^y)} = \varepsilon  + \overline{\T_y(b_r^y)} = \varepsilon + \overline{\T_y(b_i^z)} = \overline 0,
\end{align*}
a contradiction. So \(b_i^z \neq b_r^y\), thus \(\T_y(b_i^z) =\T(b_i^z)\) and \(u_{i+1} \leq l_i\).

 We have \((b^y_{l_i + 1})_{\overline{0}} = (b^y_{l_i})_{\overline{0}} + 1\), so \((b^y_{l_i +1})_{\overline 0} = (b_{i+1}^z)_{\overline 0}\). Moreover, since \(\T_y(b^y_{l_i +1}) = \T(b_{l_i}^y) < \T(b_i^z) = \T_y(b_i^z) = x_{i+1}^z\), we have that \(b_{l_i +1}^y \Rightarrow b_{i+1}^z\), so \(l_{i+1} \geq l_{i}+1>l_i \geq u_{i+1}\). Thus \(b_{i+1}^z  \in\textup{NE}\).

\item[(b)] Assume \(\varepsilon + \overline{\T_y(b_i^z)}=\overline 1\). Note \(\T_y(b_i^z) = \T_z(b_{i+1}^z)\). We have \((b_{i+1}^z)_{\overline 1} = (b_i^z)_{\overline 1}+1\). First we prove that \(u_{i+1} < l_i\). Assume this is not the case. Then by Lemma \ref{filler}, there exists a sequence \(t_0, \ldots, t_k\), where \(k=(b_{u_{i+1}}^y)_{\overline 1} - (b_{l_i}^y)_{\overline 1}-1\), such that \(l_i \leq t_0 < \cdots < t_k < u_{i+1}\), \((b_{t_j}^y)_{\overline 1} =  (b_{l_i}^y)_{\overline 1} + j\) and \( \varepsilon + \overline{\T(b_{t_j}^y)}=\overline 1\) for all \(j\). Then \(b_{t_k}^y\) is in the same column as \(b_i^z\), and \(u_i < l_i \leq t_k\), so we must have \(b_i^z \downarrow b_{t_k}^y\). Then \(x_{i+1}^z = \T_y(b^z_i) \leq \T_y(b_{t_k}^y) < \T_y(b_{u_{i+1}}^y)\), a contradiction of the \(\varepsilon\)-insertion algorithm. Thus \(u_{i+1} < l_i\).
 
If \(u_{i+1} =0\), then \(b_{i+1}^z \in \textup{NE}\) unless \(l_{i+1} = 0\) and \(\varepsilon=\overline 0\). We rule out this case by way of contradiction. Let \(u_{i+1}= l_{i+1}=0\) and \(\varepsilon=\overline 0\). Then by Lemma \ref{triple}, there is no \(m\) such that \(b_m^y \searrow b_{i+1}^z\). Then \(b_{i+1}^z \nearrow b_1^y\). Thus there is a sequence \(t_0, \ldots, t_k\), where \(k=(b_{l_{i}}^y)_{\overline 0} - 2\) such that \(1\leq t_0 < \cdots < t_k < l_i\), \((b_{t_j}^y)_{\overline 0} = 1 + j\) and \(\varepsilon+\overline{\T(b_{t_j}^y)}=\overline0\) for all \(j\). Writing \(t_{k+1}:=l_i\), there is some \(t_a\) such that \(b_{t_a}^y\) is in the same row as \(b_{i+1}^z\). But then either \(b_{i+1}^z = b_{t_a}^y\) (in which case \(x_{i+1}^z < \T_y(b_{t_a}^y)\)), or \(b_{i+1}^z \Rightarrow b_{t_a}^y\) (in which case \(x_{i+1}^z < \T_y(b_{i+1}^z) < \T_y(b_{t_a}^y)\)). But we also have \(x^z_{i+1} = \T_y(b_i^z) > \T_y(b_{l_i}^y) \geq \T_y(b_{t_a}^y)\), a contradiction of the \(\varepsilon\)-insertion algorithm.

So assume \(u_{i+1}>0\). Then by Lemma \ref{filler} there exists a sequence \(t_0, \ldots, t_k\), where \(k = (b_{l_i}^y)_{\overline 0} - (b_{u_{i+1}})_{\overline 0}-1\), such that \(u_{i+1} \leq t_0 < \cdots < t_k < l_i\), \((b_{t_j}^y)_{\overline0} = (b_{u_{i+1}})_{\overline 0}+j\) and \(\varepsilon + \overline{\T(b_{t_j}^y)}=\overline0\) for all \(j\). If \(b_i^z \Rightarrow b_{i+1}^z\), then \(l_{i+1} \geq l_i\). Otherwise there is some \(0 \leq a \leq k\) such that \(b_{t_a}^y\) is in the same row as \(b_{i+1}^z\). But since \(\T_y(b_{i+1}^z) > \T_y(b_{i}^z) > \T_y(b_{l_i}^y) > \T_y(b_{t_a}^y)\), it must be that \(b_{t_a}^y \Rightarrow b_{i+1}^z\). Then \(l_{i+1} \geq t_a \geq u_{i+1}\), and thus \(l_{i+1} > u_{i+1}\), so \(b_{i+1}^z \in \textup{NE}\).  
\end{enumerate}

Now assume that \(u_i = l_i =0\) and \(\varepsilon = \overline{1}\).  We will show that \(u_{i+1} = 0\). By way of contradiction assume \(u_{i+1}>0\). There are two cases to consider:
\begin{enumerate}
\item[(a)] Assume \(\overline{\T_y(b_i^z)}= \overline0\). Then, since \(b_1^y\) is in the first column, by Lemma \ref{filler} there exists a sequence \(t_0, \ldots, t_k\), where \(k=(b_{u_{i+1}}^y)_{\overline 1} -2\), such that \(1 \leq t_0 < \cdots < t_k < u_{i+1}\), \((b_{t_j}^y)_{\overline 1} =  1 + j\) and \(\varepsilon + \overline{\T(b_{t_j}^y)} = \overline 1\) for all \(j\). Then \(b_i^z \downarrow b_{t_k}^y\). Then \(x^z_{i+1}=\T_y(b_i^z) \leq \T_y(b_{t_k}^y)< \T_y(b_{u_{i+1}}^y)\), so by \(\overline{1}\)-insertion, \(b_{i+1}^z \downarrow b_{u_{i+1}}^y\), a contradiction.
\item[(b)] Assume \(\overline{\T_y(b_i^z)} = \overline 1\). Then, applying Lemma \ref{filler}, we have \(u_{i+1}=0\) unless \(b_i^z \Downarrow b_{i+1}^z\) and \(b_i^z = b_r^y\) for some \(r\). Then \(\T(b_{r-1}^y) = \T_y(b_r^y)\), and \(\overline{\T_y(b_r^y)} = \overline 1\), so \(b_{r-1}^y\) is in the row above \(b_r^y\), and \(b_{r}^y\NEarrow b_{r-1}^y\). Then, since \(b_1^y\) is in the first column, by Lemma \ref{filler} there exists a sequence \(t_0, \ldots, t_k\), where \(k=(b_{r-1}^y)_{\overline1} -2\), such that \(1 \leq t_0 < \cdots < t_k < r-1\), \((b_{t_j}^y)_{\overline 1} =  1 + j\) and \(\overline{\T(b_{t_j}^y)}=\overline0\) for all \(j\). Then there exist some \(t_j\) such that \(b_{t_j}^y\) is in the same column as \(b_r^y\). Then, since \(t_j<r\), we have \(b_r^y \Uparrow b_{t_j}^y\), a contradiction, since \(u_i = 0\). \end{enumerate}
This completes the proof that \(b_{i+1}^z \in \textup{NE}\) if \(b_i^z \in \textup{NE}\).

\((\impliedby)\) Now assume \(b_i^z \notin \textup{NE}(\T, \varepsilon, y)\). Let \(c_1^y, \ldots, c_{m_1}^y\) and \(c_1^z, \ldots, c_{m_2}^z\) be the bumped node sequences for the insertions \((\T' \xleftarrow{\varepsilon+\overline{1}} y)\) and  \(((\T' \xleftarrow{\varepsilon + \overline 1} y)\xleftarrow{\varepsilon + \overline 1} z)\), respectively. Then by Lemma \ref{dualconj}, \(c_j^y = (b_j^y)'\) and \(c_j^z = (b_j^z)'\) for all \(j\). But then \(u(b_j^z) = l(c_j^z)\) and \(l(b_j^z) = u(c_j^z)\) for all \(j\), so \(c_i^z \in \textup{NE}(\T', \varepsilon + \overline 1, y)\). Then, applying the `only if' direction of the claim proved above, we have \(c_{i+1}^z \in \textup{NE}(\T', \varepsilon + \overline 1, y)\). Then 
\begin{align*}
u(b_j^z) = l(c_j^z) > u(c_j^z) = l(b_j^z), \hspace{5mm} \textup{or} \hspace{5mm}
u(b_j^z) = l(c_j^z) = 0 = u(c_j^z) = l(b_j^z) \textup{ and } \varepsilon  + \overline{y} = \overline 0,
\end{align*}
so \(b_{i+1}^z \notin \textup{NE}(\T, \varepsilon, y)\), as required.

This completes the proof of the lemma when \(\T\) is a standard tableau, \(y,z \notin \T\), \(y \neq z\), and \(\overline y = \overline 0\). Now we maintain the above assumptions but consider the case \(\overline y = \overline 1\). Let \(c_1^{y^*}, \ldots, c_{m_1}^{y^*}\) and \(c_1^{z^*}, \ldots, c_{m_2}^{z^*}\) be the bumped node sequences for the insertions \((\T^* \xleftarrow{\varepsilon +\overline 1} y^*)\) and  \(((\T^* \xleftarrow{\varepsilon + \overline 1 } y^*)\xleftarrow{\varepsilon + \overline 1} z^*)\), respectively. Then \(b_i^y = c_i^{y^*}\) and \(b_i^z=c_i^{z^*}\) for all \(i\), so \(u(b_i^z) = u(c_i^{z^*})\) and \(l(b_i^z) = l(c_i^{z^*})\) for all \(i\) by Lemma \ref{dualconj}. Then, since \(\overline{y^*} = \overline{0}\), we have that, for all \(i\),
\begin{align*}
b_i^z \in \textup{NE}(\T, \varepsilon, y)
\hspace{3mm}
&\iff
\hspace{3mm}
c_i^{z^*} \in \textup{NE}(\T^*, \varepsilon + \overline{1}, y^*) \hspace{3mm}\\
&\iff
\hspace{3mm}
\begin{cases}
y^*\prec z^* \textup{ and } \varepsilon + \overline{1} = \overline{0}, \textup{ or}\\
y^*\succ  z^* \textup{ and } \varepsilon + \overline{1} = \overline{1}\\
\end{cases}\\
&\iff
\hspace{3mm}
\begin{cases}
y\succ z \textup{ and } \varepsilon = \overline{1}, \textup{ or}\\
y\prec  z \textup{ and } \varepsilon =\overline{0}.
\end{cases}
\end{align*}
This completes the proof of the lemma when \(\T\) is a standard tableau, \(y,z \notin \T\), \(y \neq z\).

Now, let \(\T\) be an arbitrary {\em semistandard} tableau, with arbitrary \(y, z \in \Xlet\). We may choose elements \(z^\bullet, y^\bullet \in \Xlet^\bullet\), and a \(\bullet\)-standardization \(\T^\bullet\) of \(\T\), such that 
\begin{enumerate}
\item \(\hat\bullet(z^\bullet) = z\) 
\item \(\hat\bullet(y^\bullet)=y\)
\item For all \(x \in \T^\bullet\) such that \(y= \hat\bullet(x)\), we have:
\begin{enumerate}
\item  \(x\prec y^\bullet\) if \(\varepsilon + \overline{y} = \overline 0\)
\item \(x\succ y^\bullet\) if \(\varepsilon + \overline{y} = \overline 1\)
\end{enumerate}
\item For all \(x \in \T^\bullet\) such that \(z= \hat\bullet(x)\), we have:
\begin{enumerate}
\item  \(x\prec z^\bullet\) if \(\varepsilon + \overline{z} = \overline 0\)
\item \(x\succ z^\bullet\) if \(\varepsilon + \overline{z} = \overline 1\)
\end{enumerate}
\item If \(z=y\), we have: 
\begin{enumerate}
\item  \(y^\bullet\prec z^\bullet\) if \(\varepsilon + \overline{z} = \overline 0\)
\item \(y^\bullet \succ z^\bullet\) if \(\varepsilon + \overline{z} = \overline 1\).
\end{enumerate}
\end{enumerate}

Then by this choice we have
\begin{align*}
b_i^z \in \textup{NE}(\T, \varepsilon, y)
\hspace{3mm}
&\iff
\hspace{3mm}
b_i^{z^\bullet} \in \textup{NE}(\T^\bullet, \varepsilon, y^\bullet)\\
& \iff 
\hspace{3mm}
\begin{cases}
y^\bullet \prec z^\bullet \textup{ and } \varepsilon = \overline{0}, \textup{ or}\\
y^\bullet \succ z^\bullet \textup{ and } \varepsilon = \overline{1}
\end{cases}\\
& \iff
\hspace{3mm}
 \begin{cases}
y \prec z \textup{ and } \varepsilon = \overline{0}, \textup{ or}\\
y\succ z \textup{ and } \varepsilon = \overline{1}, \textup{ or}\\
y=z \textup{ and } \overline{y} = \overline{0},
\end{cases}
\end{align*}
by application of Lemma \ref{standsame}. 
\end{proof}

\begin{Corollary}\label{addednodes}
Assume \(\varepsilon \in \ZZ_2\), \(y,z  \in \Xlet\), and  \(\T\) is a semistandard \((\Xlet,\lambda)\)-tableau. Then
\begin{align*}
A(\T, \varepsilon, y) \nearrow A((\T \xleftarrow{ \varepsilon } y),\varepsilon, z) \hspace{3mm} \iff \hspace{3mm} 
\begin{cases}
y\prec z \textup{ and } \varepsilon = \overline{0}, \textup{ or}\\
y\succ  z \textup{ and } \varepsilon = \overline{1}, \textup{ or}\\
y = z \textup{ and } \overline{y} = \overline{0}.
\end{cases}
\end{align*}
\end{Corollary}

\begin{proof}
Let \(b_1^y, \ldots, b^y_{m_1}\), \(b_1^z, \ldots, b_{m_2}^z\) be as in Lemma \ref{nodecontain}. By that lemma, \(b_{m_2}^z \in \textup{NE}(\T, \varepsilon, y)\) if and only if the right side holds.

\((\impliedby)\) Assume by way of contradiction that \(b_{m_2}^z \in \textup{NE}\) and \(b_{m_2}^z \nearrow b_{m_1}^y\). Then \(b_{m_1}^y\) cannot be in the same column as \(b_{m_2}^z\), else \(u_{m_2} > l_{m_2}\).

First assume \(l_{m_2}>0\). Then by Lemma \ref{filler}, there exists a sequence \(t_0, \ldots, t_k\), where \(k=(b_{m_1}^y)_{\overline 1} - (b_{l_{m_2}}^y)_{\overline 1}-1\), such that \(l_{m_2} \leq t_0 < \cdots < t_k < m_1\), \((b_{t_j}^y)_{\overline 1} =  (b_{l_{m_2}}^y)_{\overline 1} + j\) and \(\varepsilon + \overline{\T(b_{t_j}^y)} = \overline 1\) for all \(j\). Then there is some \(t_j\) such that \(b_{t_j}^y\) is in the same column as \(b_{m_2}^z\). Moreover, we have \(b_{m_2}^z \Uparrow b_{t_j}^y\), hence \(u_{m_2} \geq t_j > l_{m_2}\), a contradiction.

Now assume \(l_{m_2}=0\). Then \(u_{m_2}=0\) and \(\varepsilon + \overline{y}=\overline 1\). Then \(b_1^y\) is in the first column, and by Lemma \ref{filler}, there exists a sequence \(t_0, \ldots, t_k\), where \(k=(b_{m_1}^y)_{\overline 1} - 2\), such that \(1 \leq t_0 < \cdots < t_k < m_1\), \((b_{t_j}^y)_{\overline 1} =  1 + j\) and \(\varepsilon + \overline{\T(b_{t_j}^y)} = \overline 1\) for all \(j\).  Then there is some \(t_j\) such that \(b_{t_j}^y\) is in the same column as \(b_{m_2}^z\). Moreover, we have \(b_{m_2}^z \Uparrow b_{t_j}^y\), hence \(u_{m_2} \geq t_j > 0\), a contradiction.

\((\implies)\) Applying the `if' statement proved above to the conjugate situation (as in the proof of claim Lemma \ref{nodecontain}), we have that \(b_{m_2} \notin \textup{NE}\) implies that \(b_{m_2}^z \nearrow b_{m_1}^y\), completing the proof.
\end{proof}

\section{Super RSK correspondence}\label{SRSKsec}

\subsection{Biwords}
Given alphabets \(\Xlet\) and \(\Ylet\), we call an element of \(\Xlet \times \Ylet\) an {\em \((\Xlet,\Ylet)\)-biletter}. We call a biletter \((x,y)\) {\em mixed} if \(\overline{x} + \overline{y} = \overline{1}\). We define a total order \(\triangleleft\) on \((\Xlet,\Ylet)\)-biletters by setting  
\((x_1,y_1) \triangleleft (x_2,y_2)\) if
\begin{align*}
y_1 <_\Ylet y_2,  \hspace{10mm} \textup{or} \hspace{10mm}
y_1 = y_2,\; x_1 \prec_{\Xlet} x_2.
\end{align*}

For \(k \in \ZZ_{>0}\), we call an element \(\bw = ((x_1,y_1), \ldots, (x_k,y_k)) \in (\Xlet \times \Ylet)^k\) an {\em \((\Xlet,\Ylet)\)-biword of length \(k\)}. We say that \(\bw\) is {\em restricted} if it is multiplicity free with respect to mixed biletters; i.e. \((x_i,y_i) = (x_j, y_j)\) for \(i \neq j\) only if \(\overline{x}_i + \overline{y}_i = \overline{0}\). We say that \(\bw\) is {\em ordered} if \((x_i, y_i) \trianglelefteq (x_j,y_j)\) for all \(i \leq j\). The {\em left content} \(\textup{lcon}(\bw)\) of \(\bw\) is the multiset \(\{x_1, \ldots, x_k\}\) and the {\em right content} \(\textup{rcon}(\bw)\) of \(\bw\) is the multiset \(\{y_1, \ldots, y_k\}\).

If \(L\) is a multiset of elements of \(\Xlet\) and \(R\) is a multiset of elements of \(\Ylet\), with \(|L| = |R| = k\), we say \((L,R)\) is an \((\Xlet,\Ylet)\){\em -content pair of length \(k\)}. For an \((\Xlet, \Ylet)\)-content pair, define \(\textup{RBiw}(L,R)\) to be the set of restricted \((\Xlet, \Ylet)\)-biwords \(\bw\) with \(\textup{lcon}(\bw)=L\) and \(\textup{rcon}(\bw)=R\). Let \(\textup{RBiw}(L,R)^{\trianglelefteq} =\{ \bw \in \textup{RBiw}(L,R) \mid \bw \textup{ is ordered}\}\).  Finally, define \(\textup{Tab}(L,R)\) to be the set of pairs \((\L, \R)\) of tableaux such that \(\textup{sh}(\L)=\textup{sh}(\R)\), \(\textup{con}(\L)=L\) and \(\textup{con}(\R)=R\). Let \(\textup{SStd}(L,R) \subseteq \textup{Tab}(L,R)\) be the subset of {\em semistandard} tableau pairs.

\subsection{Super RSK algorithm}\label{RSKsec} Let \((L,R)\) be an \((\Xlet,\Ylet)\)-content pair of length \(k\).  Let \(\bw = ((x_1, y_1), \ldots, (x_k,y_k)) \in \textup{RBiw}(L,R)^{\trianglelefteq}\). We define \(\T_{\bw}^0:= \varnothing\), then for \(1 \leq i \leq k\) we inductively define \(\T_{\bw}^{i}:=(\T_{\bw}^{i-1} \xleftarrow{\overline{y}_i}x_i)\), and define \(a^{i}_{\bw}\) to be the added node of this insertion. We say \(\T_{\bw}:= \T_{\bw}^k\) is the {\em insertion tableau of \(\bw\)}. The {\em recording tableau of \(\bw\)} is the \((\Ylet,\textup{sh}(\T_{\bw}))\)-tableau \(\T^{\bw}\) defined by \(\T^{\bw}(a^i_{\bw}):=y_i\). We then define
\begin{align*}
 \textup{sRSK}(\bw) := (\T_\bw, \T^\bw).
 \end{align*}

\begin{Example}\label{RSKex}
Let \(\Xlet \) be as in Example \ref{firstex}, and take \(\Ylet = \Xlet\). Let 
\begin{align*}
L = \{\hat 1, 1, \hat 2, 2, \hat 3, \hat 3, \hat 3, 3\} \hspace{5mm}\textup{ and }\hspace{5mm}R = \{\hat 1, \hat 2, 2, 2, \hat 3, \hat 3, 3, 3\}
\end{align*}
 be multisets of letters. Then \((L,R)\) is an \((\Xlet, \Ylet)\)-content pair of length 8. Let \(\bw\) be the biword
\begin{align*}
\bw = ((\hat 3, \hat 1), (1, \hat 2), (2,2), (3,2), (\hat 3, \hat 3), (\hat 3, \hat 3), (\hat 2, 3), (\hat 1, 3)).
\end{align*}
Then \(\bw \in \textup{RBiw}(L,R)^{\trianglelefteq}\), and \(\textup{sRSK}(\bw)\) yields the tableaux:
\begin{align*}
\T_\bw=\ytableausetup{centertableaux}
\begin{ytableau}
 \hat 1 & \hat 2 & \hat 3 & 3 \\
1 & 2 & \hat 3 \\
\hat 3
\end{ytableau}
\hspace{10mm}
\T^\bw=\ytableausetup{centertableaux}
\begin{ytableau}
 \hat{1} & 2 & 2 & \hat{3} \\
 \hat{2} & 3 & 3\\
 \hat{3} 
\end{ytableau}
\end{align*}
\end{Example}

\begin{Theorem}[Super RSK correspondence]\label{SuperRSK}
Let \((L,R)\) be an \((\Xlet, \Ylet)\)-content pair. The map \(\textup{sRSK}:\bw \mapsto (\T_\bw, \T^\bw)\) defines a bijection \(\textup{RBiw}(L,R)^{\trianglelefteq} \to \textup{SStd}(L,R)\).
\end{Theorem}

\begin{proof}
Let \(|L|=|R|=k\), and \(\bw = ((x_1,y_1), \ldots, (x_k,y_k)) \in \textup{RBiw}(L,R)^{\trianglelefteq}\).  We have that \(\T_\bw\) is semistandard by inductive application of Lemma \ref{semisemi} . Define \(\T^\bw_i\) by \(\T^\bw_i(a_\bw^j):=y_j\) for all \(1 \leq j \leq i\). By induction, assume:
\begin{enumerate}
\item \(\T^\bw_i\) is semistandard
\item If \(r<s\leq i\), \(y_r= y_s \), and \(\overline{y}_r = \overline 0\), then \(a_\bw^r \nearrow a_\bw^s\)
\item If \(r<s\leq i\), \(y_r = y_s\), and \(\overline{y}_r = \overline 1\), then \(a_\bw^s \nearrow a_\bw^r\).

\end{enumerate}
If \(y_{i+1} > y_i\), then \(\T_{i+1}^{\bw}\) automatically satisfies (i)-(iii). Assume \(y_{i+1} = y_i\).  Then either \(x_{i} \prec x_{i+1}\) or \(x_{i} = x_{i+1}\) and \(\overline{y}_{i+1} + \overline{x}_{i+1} = \overline 0\). Note that \(\T_{i+1}^\bw\) is non-decreasing since the upper row of \(\bw\) is non-decreasing. There are two cases:

\begin{enumerate}
\item[(a)] Assume \(\overline{y}_{i+1} = \overline 0\). Then (iii) holds, and \(\T^\bw_i\) is column-strict with respect to odd letters. Moreover by Corollary \ref{addednodes}, \(a_\bw^{i} \nearrow a_\bw^{i+1}\), so (ii) holds, and \(\T^\bw_i\) is row-strict with respect to even letters.

\item[(b)] Assume \(\overline{y}_{i+1} = \overline 1\). Then (ii) holds, and \(\T^\bw_i\) is row-strict with respect to even letters. Moreover by Corollary \ref{addednodes}, \(a_\bw^{i+1} \nearrow a_\bw^{i}\), so (iii) holds, and \(\T^\bw_i\) is column-strict with respect to odd letters.
\end{enumerate}

Thus, by induction \(\T_k^\bw = \T^\bw\) is semistandard. Thus \(\textup{sRSK}(\bw) \in \textup{SStd}(L,R)\).

Now let \((\L, \R) \in \textup{SStd}(L,R)\). We define \(\L_k=\L\), \(\R_k = \R\), and then for \(1 \leq i \leq k\) inductively define \(\R_{i-1}\) and \(\L_{i-1}\) in the following manner. Define \(y_i \in \Ylet\) to be the \(<\)-maximal element of \(\R_{i}\). If \(\overline{y}_i = \overline0\) (resp. \(\overline{y}_i =\overline1\)), let \(u_i\) be the rightmost (resp. bottommost) node in \(\R_{i}\) such that \(\R_{i}(u_i)=y_i\). Let \(\R_{i-1} = \R_i \backslash \{u_i\}\). Let \(\L_{i-1}= (\L_i \xrightarrow{ \overline{y}_i} u_i)\), and let \(x_i\) be the extracted letter. Then define 
\begin{align*}
\textup{sRSK}^*(\L,\R) = ((x_1,y_1),\ldots, (x_k,y_k)) 
\end{align*}

Let \(\bw:= \textup{sRSK}^*(\L,\R)\). By construction and Lemma \ref{inverses} we have that \(\T^j_\bw = \L_j\), \(\T_j^\bw = \R_j\), and \(u_j = a_\bw^j\) for all \(1 \leq j \leq k\). We argue by induction on \(i\) that 
\begin{align*}
\bw_i := ((x_1,y_1), \ldots, (x_i,y_i))
\end{align*}
is an {\em ordered} restricted \((\Xlet, \Ylet)\)-biword. By construction, \(y_{i+1} \geq y_i\), so \(\bw_{i+1}\) is an ordered restricted biword if \(y_{i+1} \neq y_i\). Assume \(y_i = y_{i+1}\). Then there are two cases. 
\begin{enumerate}

\item[(a)]  Assume \(\overline{y}_{i+1} =\overline 0\). If \(x_i \succ x_{i+1}\) or \(x_i = x_{i+1}\) and \(\overline{x}_i = \overline 1\), then by Corollary \ref{addednodes}, \(u_{i+1} = a_{\bw}^{i+1} \nearrow a_{\bw}^i = u_i\), which by the choice of \(u_{i+1}\) implies that \(u_i \Downarrow u_{i+1}\), \(\R(u_i) = \R(u_{i+1})\), and \(\overline{\R(u_i)}=\overline 0\), a contradiction, since \(\R\) is semistandard.
\item[(b)]  Assume \(\overline{x}_{i+1} =\overline1\). If \(x_i \succ x_{i+1}\) or \(x_i = x_{i+1}\) and \(\overline{x}_i = \overline{0}\), then by Corollary \ref{addednodes}, \(u_{i} = a_{\bw}^{i} \nearrow a_{\bw}^{i+1} = u_{i+1}\), which by the choice of \(u_{i+1}\) implies that \(u_i \Rightarrow u_{i+1}\), \(\R(u_i) = \R(u_{i+1})\), and \(\overline{\R(u_i)} = \overline1\), a contradiction, since \(\R\) is semistandard.
\end{enumerate}
Thus \(\bw_{i+1}\) is an ordered restricted biword. Thus by Lemma \ref{inverses}, \(\textup{sRSK}\) and \(\textup{sRSK}^*\) are mutual inverses on \(\textup{RBiw}(L,R)^{\trianglelefteq}\) and \(\textup{SStd}(L,R)\).
\end{proof}

\begin{Remark}
When \(\Xlet = \Ylet= \NN\), where \(<\) is the usual order on integers and every element is of even parity, the super RSK correspondence of Theorem \ref{SuperRSK} reduces to the classical RSK correspondence \cite{Knuth}.
\end{Remark}

\begin{Remark}\label{RSKinLitBSV}
As noted in \S\ref{introsec}, the existence of a bijection between the sets in Theorem \ref{SuperRSK} was proved by Bonetti, Senato, and Venezia \cite{BSV}, using a more straightforward insertion algorithm which yields a different bijection than the one in Theorem \ref{SuperRSK}. Our motivation in presenting this new bijection is in the direction of fully generalizing the symmetry property of classical RSK, which we do in \S\ref{symsec}.
\end{Remark}

\begin{Remark}\label{RSKinLitSW}
In \cite[\S 3]{SW}, Shimozono and White present a close relative to our super RSK correspondence---the algorithm they use to construct the upper and lower tableaux of an \((\Xlet, \Ylet)\)-biword (called in their paper a {\em doubly-colored biword}) is very similar in spirit to the super RSK algorithm presented here (see Remark \ref{insertinLitSW}). However, they work with the set of all (not just restricted) biwords, and their semistandard tableaux are defined to be row-weak and column-strict for {\em both} parities. Consequently, the fact that a bijective correspondence exists between these objects (as noted in \cite[Theorem 22]{SW}) can be deduced from classical RSK correspondence, while this is not true of the correspondence in Theorem \ref{SuperRSK}, which involves distinct (and distinctly-sized) sets of combinatorial objects.
\end{Remark}

Note that \(\textup{RBiw}(L,R)^{\trianglelefteq}\) is a set of orbit representatives for \(\textup{RBiw}(L,R)\) under the action of the symmetric group \(\mathfrak{S}_k \). For \(\bw  \in \textup{RBiw}(L,R)\), write \(\bw^{\trianglelefteq}\) for the unique element of \(\textup{RBiw}(L,R)^{\trianglelefteq}\) which belongs to the \(\mathfrak{S}_k\)-orbit of \(\bw\). By precomposing with the function \(\bw \mapsto \bw^{\trianglelefteq}\), we may extend \(\textup{sRSK}\) to a function \(\textup{sRSK}:\textup{RBiw}(L,R) \to \textup{SStd}(L,R)\) which is constant on \(\mathfrak{S}_k\)-orbits.

\section{Symmetry}\label{symsec}
In this section we prove that the super RSK algorithm defined in \S\ref{RSKsec} satisfies the symmetry property that holds for the classical RSK algorithm. In this section we assume that \((L,R)\) is an \((\Xlet, \Ylet)\)-content pair of length \(k\).

\subsection{Inversion}
Let \(\bw = ((x_1,y_1), \ldots, (x_k,y_k)) \in \textup{RBiw}(L,R)^{\trianglelefteq}\). Then there is a unique biword \(\bw^{\textup{inv}} \in \textup{RBiw}(R,L)^{\trianglelefteq}\) which consists of the biletters \((y_1,x_1), \ldots, (y_k,x_k)\). I.e., we construct \(\bw^{\textup{inv}}\) by swapping the entries of the biletters in \(\bw\), then reordering the biletters according to the ordering on biletters. We refer to \(\bw^\textup{inv}\) as the {\em inversion} of \(\bw\). For \((\L, \R) \in \textup{Tab}(L,R)\), write \((\L,\R)^\textup{inv}:=(\R,\L) \in \textup{Tab}(R,L)\).

\subsection{Standardizing biwords}
We will say a biword is {\em standard} if no letter occurring in the biword has multiplicity greater than one. For a multiset \(L\) of letters in \(\Xlet\), we define
\begin{align*}
L^\bullet = \{ x^{(i)}  \in \Xlet^{\bullet} \mid x \in \Xlet, 1 \leq i \leq \textup{mult}_L(x)\},
\end{align*}
where \(\textup{mult}_L(x)\) is the multiplicity of \(x\) in \(L\).

\begin{Definition}\label{wstd} Let \((L,R)\) be an \((\Xlet, \Ylet)\)-content pair of length \(k\), and let \(\bw \in \textup{RBiw}(L,R)^{\trianglelefteq}\). We construct a related biword in \(\textup{RBiw}(L^\bullet, R^\bullet)^{\trianglelefteq}\) as follows.

Label the distinct elements of \(L\) such that \(x_1 \prec x_2 \prec \cdots \prec x_s\), and label the distinct elements of \(R\) such that \(y_1 \prec y_2 \prec \cdots \prec y_t\). Let \(\ell_{i,j}\) be the multiplicity of \((x_i,y_j)\) in \(\bw\). Then define \(\bw^\bullet\) to be the unique biword in \(\textup{RBiw}(L^\bullet,R^\bullet)^{\trianglelefteq}\) which consists of the biletters
\begin{align*}
(x_i^{(\ell_{i,1} + \cdots + \ell_{i,j-1} + m)},y_j^{(\ell_{1,j} + \cdots + \ell_{i-1,j} + m)}) \textup{ for }i \in [1,s], j \in [1,t], m \in [1,\ell_{i,j}], \overline{y}_j = \bar 0,
\end{align*}
and
\begin{align*}
(x_i^{(\ell_{i,1} + \cdots + \ell_{i,j-1} + m)},y_j^{(\ell_{1,j} + \cdots + \ell_{i,j} +1-m)}) \textup{ for }i \in [1,s], j \in [1,t], m \in [1,\ell_{i,j}], \overline{y}_j = \bar 1.
\end{align*}
We call \(\bw^\bullet\) the {\em \(\bullet\)-standardization of \(\bw\)}.
\end{Definition}

Note that by construction, \(\bw^\bullet\) is a standard biword. Define \(\textup{SRBiw}(L^\bullet,R^\bullet)^{\trianglelefteq}\) as the set of {\em standard} restricted biwords in \(\textup{RBiw}(L^\bullet,R^\bullet)^{\trianglelefteq}\). Let \(\bullet: \textup{RBiw}(L,R)^{\trianglelefteq} \to \textup{SRBiw}(L^\bullet,R^\bullet)^{\trianglelefteq}\) be the map defined by \(\bw \mapsto \bw^\bullet\). Let \(\hat{\bullet}:\textup{SRBiw}(L^\bullet,R^\bullet)^{\trianglelefteq} \to \textup{Biw}(L,R)\) be given by `forgetting superscripts'. By definition of the orders on biletters we have that \(\hat{\bullet} \circ \bullet\) is the identity on \(\textup{RBiw}(L,R)^{\trianglelefteq}\).

\begin{Lemma}\label{bullinv}
For \(\bw \in \textup{RBiw}(L,R)^{\trianglelefteq}\), we have \((\bw^\bullet)^\textup{inv} = (\bw^\textup{inv})^\bullet\). 
\end{Lemma}

\begin{proof}
If \(\ell_{i,j}\) is the multiplicity of \((x_i,y_j)\) in \(\bw\), then \(\ell_{i,j}\) is the multiplicity of \((y_j,x_i)\) in \(\bw^\textup{inv}\). Then \((\bw^\bullet)^\textup{inv}\) consists of the biletters
\begin{align*}
(y_j^{(\ell_{1,j} + \cdots + \ell_{i-1,j} + m)},x_i^{(\ell_{i,1} + \cdots + \ell_{i,j-1} + m)}) \textup{ for }i \in [1,s], j \in [1,t], m \in [1,\ell_{i,j}], \overline{y}_j = \bar 0,
\end{align*}
and
\begin{align*}
(y_j^{(\ell_{1,j} + \cdots + \ell_{i,j} +1-m)},x_i^{(\ell_{i,1} + \cdots + \ell_{i,j-1} + m)}) \textup{ for }i \in [1,s], j \in [1,t], m \in [1,\ell_{i,j}], \overline{y}_j = \bar 1.
\end{align*}
while \((\bw^\textup{inv})^\bullet\) consists of the biletters
\begin{align*}
(y_j^{(\ell_{1,j} + \cdots + \ell_{i-1,j} + m)},x_i^{(\ell_{i,1} + \cdots + \ell_{i,j-1} + m)}) \textup{ for }j \in [1,t], i \in [1,s], m \in [1, \ell_{i,j}], \overline{x}_i = \bar 0,
\end{align*}
and
\begin{align*}
(y_j^{(\ell_{1,j} + \cdots + \ell_{i-1,j} + m)},x_i^{(\ell_{i,1} + \cdots + \ell_{i,j} + 1-m)}) \textup{ for }j \in [1,t], i \in [1,s], m \in [1, \ell_{i,j}], \overline{x}_i = \bar 1.
\end{align*}

Assume that \((y_j^{(a)},x_i^{(b)})\) is a biletter of \((\bw^\bullet)^\textup{inv}\), for some \(j \in [1,t]\), \(i \in [1,s]\), and \(a,b \in \Z_{>0}\). We will show that \((y_j^{(a)}, y_i^{(b)})\) is a biletter in \((\bw^\textup{inv})^\bullet\) as well. We consider three cases:
\begin{enumerate}
\item[(a)] Assume that \(\overline{y}_j + \overline{x}_i = \bar 1\). Then \(\ell_{i,j} = 1\) since \(\bw\) is a restricted biword. Then \(a = \ell_{1,j} + \cdots + \ell_{i,j}\) and \(b = \ell_{i,1} + \cdots + \ell_{i,j}\), so the claim follows.

\item[(b)] Assume that \(\overline{y}_j = \overline{x}_i = \bar 0\). Then \(a=\ell_{1,j} + \cdots + \ell_{i-1,j} +m\) and \(b=\ell_{i,1} + \cdots + \ell_{i,j-1} +m\), for some \(m \in [1,\ell_{i,j}]\), so the claim follows.
\item[(c)] Assume that \(\overline{y}_j = \overline{x}_i = \bar 1\). Then \(a=\ell_{1,j} + \cdots + \ell_{i,j} +1-m\) and \(b=\ell_{i,1} + \cdots + \ell_{i,j-1} +m\), for some \(m \in [1,\ell_{i,j}]\). But then, setting \(n=\ell_{i,j}-m+1\), we have that \(n \in [1,\ell_{i,j}]\), and \(a=\ell_{1,j} + \cdots + \ell_{i-1,j} + n\) and \(b=\ell_{i,1} + \cdots + \ell_{i,j} +1 -n\), so the claim follows.
\end{enumerate}
Then in any case, every biletter of \((\bw^\bullet)^\textup{inv}\) appears in \((\bw^\textup{inv})^\bullet\). Since the letters in \(\bw^\bullet\) appear with multiplicity one, and the biwords \((\bw^\bullet)^\textup{inv}\) and  \((\bw^\textup{inv})^\bullet\) are ordered and have the same length, this completes the proof of the lemma.
\end{proof}

We may extend the `forget superscripts' map \(\hat \bullet: \Xlet^\bullet \to \Xlet\) to a map \( \hat\bullet:\textup{SStd}(L^\bullet, R^\bullet) \to \textup{Tab}(L,R)\) via `forgetting superscripts' of all entries in the tableaux.

\begin{Lemma}\label{bullXi}
For \(\bw \in \textup{RBiw}(L,R)^{\trianglelefteq}\) we have \(\hat\bullet(\textup{sRSK}(\bw^\bullet)) = \textup{sRSK}(\bw).\)
\end{Lemma}

\begin{proof}
Recall from Definition \ref{wstd} that we label the distinct elements of \(L\) such that \(x_1 \prec x_2 \prec \cdots \prec x_s\), and label the distinct elements of \(R\) such that \(y_1 \prec y_2 \prec \cdots \prec y_t\). Assume that \((x_i^{(a)}, y_j^{(b)}) \triangleleft (x_i^{(c)}, y_{j'}^{(d)})\) are biletters in \(\bw^\bullet\). We first prove that \(x_i^{(a)} \prec x_i^{(c)}\) if and only if \(\overline{x}_i +\overline{y}_{j'}= \bar 0\). There are three cases to consider:
\begin{enumerate}
\item[(a)] Assume that \(j<j'\). Then we have \(y_j \prec y_{j'}\) and \(y_j^{(b)} < y_{j'}^{(d)}\), which implies that \(y_j < y_{j'}\). Then \(\overline{y}_{j'} = \bar 0\). By the definition of the \(\bullet\)-standardization \(\bw^\bullet\), we have, for some \(m \in [1,\ell_{i,j}]\), \(n \in [1,\ell_{i,j'}]\), 
\begin{align*}
a= \ell_{i,1} + \cdots + \ell_{i,j-1} +m < \ell_{i,1} + \cdots + \ell_{i,j-1} + \cdots + \ell_{i,j'-1}+n =c.
\end{align*}
Thus \(x_i^{(a)} < x_i^{(c)}\). The claim follows.
\item[(b)] Assume that \(j>j'\). Then we have \(y_j \succ y_{j'}\) and \(y_j^{(b)} < y_{j'}^{(d)}\), which implies that \(y_j > y_{j'}\). Then \(\overline{y}_{j'} = \bar 1\). Then by the definition of the \(\bullet\)-standardization \(\bw^\bullet\), we have, for some \(m \in [1,\ell_{i,j}]\), \(n \in [1,\ell_{i,j'}]\), 
\begin{align*}
a= \ell_{i,1} + \cdots + \ell_{i,j'-1} + \cdots + \ell_{i,j-1} +m > \ell_{i,1} + \cdots + \ell_{i,j'-1}+n =c.
\end{align*}
Thus \(x_i^{(a)} > x_i^{(c)}\). The claim follows.
\item[(c)] Assume that \(j=j'\) and \(\bar{y}_{j'} = \bar 0\). Then \(b<d\) since \(y_j^{(b)} = y_{j'}^{(b)}<y_{j'}^{(d)}\), and we have, for some \(m,n \in [1,\ell_{i,j'}]\),
\begin{align*}
\ell_{1,j'} + \cdots + \ell_{i-1,j'} +m = b < d = \ell_{1,j'} + \cdots + \ell_{i-1,j'} + n,
\end{align*}
so \(m<n\). Then we have
\begin{align*}
a= \ell_{i,1} + \cdots + \ell_{i,j'-1} +m < \ell_{i,1} + \cdots + \ell_{i,j'-1} + n = c,
\end{align*}
so \(x_i^{(a)} < x_i^{(c)}\). The claim follows.
\item[(d)] Assume that \(j=j'\) and \(\bar{y}_{j'} = \bar 1\). Then \(b<d\) since \(y_j^{(b)} = y_{j'}^{(b)}<y_{j'}^{(d)}\), and we have, for some \(m,n \in [1,\ell_{i,j'}]\),
\begin{align*}
\ell_{1,j'} + \cdots + \ell_{i,j'} +1-m = b < d = \ell_{1,j'} + \cdots + \ell_{i,j'} + 1-n,
\end{align*}
so \(m>n\). Then we have
\begin{align*}
a= \ell_{i,1} + \cdots + \ell_{i,j'-1} +m > \ell_{i,1} + \cdots + \ell_{i,j'-1} + n = c,
\end{align*}
so \(x_i^{(a)} > x_i^{(c)}\). The claim follows.
\end{enumerate}
Thus in any case, the claim follows. 

Let \({\tt T}_{\bw}^i\) (resp. \({\tt T}_{\bw^\bullet}^i\)) be the \(i\)th insertion tableaux in the Super RSK algorithm applied to \(\bw\) (resp. \(\bw^\bullet\)), and let \(a_{\bw}^i\) (resp. \(a_{\bw^\bullet}^i\)) be the added node of this insertion, using notation in \S\ref{RSKsec}. By induction, assume that \(\hat\bullet({\tt T}^i_{\bw^\bullet})\) is a \(\bullet\)-standardization of \( {\tt T}^i_{\bw}\), and \(a_{\bw^\bullet}^i = a_{\bw}^i\), for all \(i <n\). If \((x_i^{(c)}, y_{j'}^{(d)})\) is the \(n\)th biletter in \(\bw^\bullet\), then \((x_i,y_{j'})\) is the \(n\)th biletter in \(\bw\). If \(\overline{y}_{j'} + \overline{x}_i= \bar 0\), then, by the above claim, we have \(x_i^{(c)} \succ z\) for every \(z \in {\tt T}_{\bw^\bullet}^{n-1}\) such that \(\hat \bullet(z) = x_i\). On the other hand if \(\overline{y}_{j'} + \overline{x}_i= \bar 1\), then, by the above claim, we have \(x_i^{(c)} \prec z\) for every \(z \in {\tt T}_{\bw^\bullet}^{n-1}\) such that \(\hat \bullet(z) = x_i\). Therefore by Lemma \ref{standsame}, it follows that \({\tt T}_{\bw^\bullet}^n= ({\tt T}_{\bw^\bullet}^{n-1} \xleftarrow{\overline{y}_{j'}} x_i^{(c)}) \) is a \(\bullet\)-standardization of \( ({\tt T}_{\bw}^{n-1} \xleftarrow{\overline{y}_{j'}} x_i) = {\tt T}_{\bw}^n\), and \(a_{\bw^\bullet}^n = a_{\bw}^n\), as desired. Thus \(\hat \bullet({\tt T}_{\bw^\bullet}) = {\tt T}_{\bw}\) and \(\hat \bullet({\tt T}^{\bw^\bullet}) = {\tt T}^{\bw}\), so \(\hat \bullet(\textup{sRSK}(\bw^\bullet)) = \textup{sRSK}(\bw)\), proving the lemma.
\end{proof}

\subsection{Symmetry}  As noted in Remark \ref{RSKinLitSW}, Shimizono and White \cite{SW} define a super-analogue of the RSK algorithm which is identical to the super RSK algorithm presented here when restricted to standard biwords. Thus their symmetry result proves a special case of the symmetry of the \(\textup{sRSK}\) map:

\begin{Lemma}\label{SWsym}
If \(\bw \in \textup{RBiw}(L,R)^{\trianglelefteq}\) is a standard biword, then we have
\begin{align*}
\textup{sRSK}(\bw^{\textup{inv}}) = (\textup{sRSK}(\bw))^\textup{inv}.
\end{align*}
\end{Lemma}
\begin{proof}
This follows from \cite[Theorem 21(3),(6)]{SW}.
\end{proof}

Now we extend this result to the general case.

\begin{Theorem} The following is a commuting diagram:
\begin{center}
\begin{tikzcd}[back line/.style={densely dotted}, row sep=1.7em, column sep=1.7em]
& \textup{RBiw}(L,R)^{\trianglelefteq} \ar{dl}[swap, sloped, near start]{ \textup{inv}\hspace{1mm}  } \ar{rr}{\bullet} \ar[back line]{dd}[near end]{\textup{sRSK}} 
  & & \textup{SRBiw}(L^\bullet,R^\bullet)^{\trianglelefteq} \ar{dd}{\textup{sRSK}} \ar{dl}[swap, sloped ,near start]{\textup{inv} \hspace{2mm}} \\
\textup{RBiw}(R,L)^{\trianglelefteq} \ar[crossing over]{rr}[near start]{\bullet} \ar{dd}[swap]{\textup{sRSK}} 
  & & \textup{SRBiw}(R^\bullet,L^\bullet)^{\trianglelefteq}  \\
& \textup{Tab}(L,R)  \ar[back line, sloped, below]{dl}{\textup{inv}} 
  & &\textup{SStd}(L^\bullet, R^\bullet)\ar[sloped]{dl}{\textup{inv}} \ar[swap,back line,near start]{ll}{\hat\bullet}\\
\textup{Tab}(R,L)  & &  \textup{SStd}(R^\bullet, L^\bullet) \ar[crossing over, leftarrow]{uu}[near start]{\textup{sRSK}} \ar{ll}[below]{\hat\bullet}
\end{tikzcd}
\end{center}
\end{Theorem}
\begin{proof}
The top face commutes by Lemma \ref{bullinv}. The bottom face commutes since `forgetting superscripts' then swapping tableaux clearly yields the same result as swapping tableaux and then `forgetting superscripts'. The front and back faces commute by Lemma \ref{bullXi}. The right face commutes by Lemma \ref{SWsym}. Thus we have
\begin{align*}
\textup{sRSK} \circ \textup{inv} &= \hat \bullet \circ \textup{sRSK} \circ \bullet\circ \textup{inv}
= \hat \bullet \circ \textup{sRSK} \circ \textup{inv} \circ \bullet\\
&= \hat \bullet \circ \textup{inv} \circ \textup{sRSK} \circ \bullet
= \textup{inv} \circ \hat \bullet \circ \textup{sRSK} \circ \bullet
= \textup{inv} \circ \textup{sRSK},
\end{align*}
so the left face commutes, proving the theorem.
\end{proof}

\begin{Corollary}[Super RSK symmetry]\label{corsym}
For all \(\bw \in \textup{RBiw}(L,R)^{\trianglelefteq}\) we have
\begin{align*}
\textup{sRSK}(\bw^{\textup{inv}}) = (\textup{sRSK}(\bw))^\textup{inv}.
\end{align*}
\end{Corollary}

\begin{Example}\label{bigex}
As in Example \ref{RSKex}, take the alphabet 
\begin{align*}
\Xlet = \Ylet = \{\hat 1 < 1 < \hat 2 < 2 < \hat 3 < 3\},
\end{align*}
where odd parity letters are indicated by carets, and the restricted biword
\begin{align*}
\bw = ((\hat 3, \hat 1), (1, \hat 2), (2,2), (3,2), (\hat 3, \hat 3), (\hat 3, \hat 3), (\hat 2, 3), (\hat 1, 3)).
\end{align*}
The inversion of \(\bw\) is
\begin{align*}
\bw^\textup{inv} = ((3 , \hat 1), ( \hat 2 , 1), ( 3, \hat 2), (2  , 2), (\hat 3 , \hat 3), ( \hat 3, \hat 3), (  \hat 1, \hat 3) , ( 2, 3)).
\end{align*}
We have
\begin{align*}
\textup{sRSK}(\bw)
&=\left(
\ytableausetup{centertableaux,,boxsize=1.6em}
\;\;\;
\begin{ytableau}
 \hat 1 & \hat 2 & \hat 3 & 3 \\
1 & 2 & \hat 3 \\
\hat 3
\end{ytableau}
\hspace{5mm}
,
\hspace{5mm}
\ytableausetup{centertableaux}
\begin{ytableau}
 \hat{1} & 2 & 2 & \hat{3} \\
 \hat{2} & 3 & 3\\
 \hat{3} 
\end{ytableau}
\;\;\;
\right),\\
\\
\textup{sRSK}(\bw^\textup{inv})
&=
\left(\;\;\;
\ytableausetup{centertableaux}
\begin{ytableau}
 \hat{1} & 2 & 2 & \hat{3} \\
 \hat{2} & 3 & 3\\
 \hat{3} 
\end{ytableau}
\hspace{5mm},
\hspace{5mm}
\ytableausetup{centertableaux}
\begin{ytableau}
 \hat 1 & \hat 2 & \hat 3 & 3 \\
1 & 2 & \hat 3 \\
\hat 3
\end{ytableau}
\;\;\;\right)
,
\end{align*}
so \(\textup{sRSK}(\bw^\textup{inv}) = \textup{sRSK}(\bw)^\textup{inv}\), as expected.

By way of comparison, consider the super-RSK algorithm (label it  \(\textup{sRSK}_{\textup{BSV}}\) to differentiate it from the algorithm in this paper) defined by Bonetti, Senato and Venezia \cite{BSV}. When we apply \(\textup{sRSK}_{\textup{BSV}}\) \(\bw\) and \(\bw^\textup{inv}\), we get (after reordering the biletters to agree with their combinatorial setup):
\begin{align*}
\textup{sRSK}_\textup{BSV}(\bw)
&=
\left(\;\;\;
\ytableausetup{centertableaux}
\begin{ytableau}
 \hat{1} & \hat 2 &3 \\
 1 & 2 \\
 \hat 3\\
 \hat 3\\
 \hat 3 
\end{ytableau}
\hspace{5mm},
\hspace{5mm}
\ytableausetup{centertableaux}
\begin{ytableau}
\hat 1 & \hat 2 & 2\\
2 & 3\\
\hat 3\\
\hat 3 \\ 
3
\end{ytableau}
\;\;\;\right)
,\\
\\
\textup{sRSK}_\textup{BSV}(\bw^\textup{inv})
&=
\left(\;\;\;
\ytableausetup{centertableaux}
\begin{ytableau}
\hat 1 & \hat 2 & 2 & 2\\
\hat 3 & 3 & 3\\
\hat 3
\end{ytableau}
\hspace{5mm},
\hspace{5mm}
\ytableausetup{centertableaux}
\begin{ytableau}
\hat 1 & \hat 2 & \hat 3 & 3\\
1 & 2 & \hat 3\\
\hat 3
\end{ytableau}
\;\;\;\right)
,
\end{align*}
so evidently symmetry does not generally hold for \(\textup{sRSK}_\textup{BSV}\). 

Finally consider the super-RSK algorithm (label it \(\textup{sRSK}_\textup{SW}\)) defined by Shimozono and White \cite{SW}. When we apply \(\textup{sRSK}_{\textup{SW}}\) to \(\bw\) and \(\bw^\textup{inv}\), we get (again after reordering biletters):
\begin{align*}
\textup{sRSK}_\textup{SW}(\bw)
&=
\left(\;\;\;
\ytableausetup{centertableaux}
\begin{ytableau}
\hat 1 & \hat 2 & \hat 3 & \hat 3 & \hat 3\\
1 & 2 & 3
\end{ytableau}
\hspace{5mm},
\hspace{5mm}
\ytableausetup{centertableaux}
\begin{ytableau}
\hat 1 & 2 & 2 & \hat 3 & \hat 3\\
\hat 2 & 3 & 3
\end{ytableau}
\;\;\;\right)
,\\
\\
\textup{sRSK}_\textup{SW}(\bw^\textup{inv})
&=
\left(\;\;\;
\ytableausetup{centertableaux}
\begin{ytableau}
\hat 1 & 2 & 2 & \hat 3 & \hat 3\\
\hat 2 & 3 & 3
\end{ytableau}
\hspace{5mm},
\hspace{5mm}
\ytableausetup{centertableaux}
\begin{ytableau}
\hat 1 & \hat 2 & \hat 3 & \hat 3 & \hat 3\\
1 & 2 & 3
\end{ytableau}
\;\;\;\right)
,
\end{align*}
so \(\textup{sRSK}_\textup{SW}(\bw^\textup{inv}) = \textup{sRSK}_\textup{SW}(\bw)^\textup{inv}\), as expected, given \cite[Theorem 21]{SW}. Note however that the tableaux output by the \(\textup{sRSK}_\textup{SW}\) algorithm are row-weak and column-strict with respect to {\em both} parities, a different flavor of `semistandard' than the notion defined in this paper.

\end{Example}


\begin{thebibliography}{ABC}


\bibitem[BSV]{BSV}
F. Bonetti, D. Senato, and A. Venezia. The Robinson-Schensted correspondence for the fourfold
algebra. {\em Boll. Unione Mat. Ital.  VII Ser. B 2} {\bf 3} (1998) 541--554. 

\bibitem[CPT]{CPT}
S. Clark, Y.N. Peng and S.K. Thamrongpairoj, Super tableaux and a branching rule for the general linear Lie superalgebra, {\em Linear Multilinear A.} {\bf 63} (2015) 274.

\bibitem[DR]{DR}
J. Du and H. Rui, Quantum Schur superalgebras and Kazhdan-Lusztig combinatorics, {\em J. Pure Appl. Algebra} {\bf 215} (2011), 2715Ð2737.

\bibitem[F]{Fulton}
W. Fulton, Young tableaux, London Math. Soc. Stud. Texts {\bf 35}, Cambridge University Press, Cambridge, 1997.

\bibitem[GRS]{GRS}
F. D. Grosshans, G. C. Rota, J. A. Stein, Invariant theory and superalgebras, {\em Reg. Conf. Ser. Math.} {\bf 69}, (1987)

\bibitem[H]{Haiman}
 M. D. Haiman, On mixed insertion, symmetry, and shifted Young tableaux, {\em J. Combin. Theory Ser. A}, {\bf 50} (1989), pp. 196--225.
 
 \bibitem[K]{Knuth}
 D. Knuth, Permutations, matrices, and generalized Young tableaux, {\em Pacific J. of
Math.} {\bf 34} (1970), pp. 709--727.

\bibitem[LNS]{LNS}
R. La Scala, V. Nardozza and D. Senato. Super RSK-algorithms and super
plactic monoid. {\em Internat. J. Algebra Comput. 16} {\bf 2}, (2006), 377--396.

\bibitem[MZ]{MZ}
F. Marko, A. N. Zubkov, A note on bideterminants for Schur superalgebras, {\em J. Pure Appl. Algebra}, {\bf 215} (2011), 2223--2230.


\bibitem[R]{Robinson}
G. de B. Robinson, On the Representations of the Symmetric Group, {\em Amer. J. Math.} {\bf 60} (1938): 745--760.

\bibitem[S]{Schensted}
C. Schensted, Longest increasing and decreasing subsequences, {\em Canad. J. Math.} {\bf 13} (1961), 179--191.

\bibitem[SW]{SW}
M. Shimozono and D. E. White, A Color-to-Spin Domino Schensted Algorithm, {\em Electron. J. Combin.} {\bf 8} (2001), 50 pp.

\bibitem[V]{Viennot}
G. Viennot, Une forme g\'eom\'etrique de la correspondance de Robinson-Schensted, Combinatoire et repr\'esentation du groupe sym\'etrique, Lecture Notes in Math. {\bf 579}, Springer (1977), 29--58.

\iffalse{

\bibitem[EK]{EK1}
A.\ Evseev and A.\ Kleshchev, Turner doubles and generalized Schur algebras, {\tt arXiv:1603.03840}.

\bibitem[G]{GreenCod}
J. A. Green, Combinatorics and the Schur algebra, {\em J. Pure Appl. Algebra} {\bf 88} (1993), 89--106.

\bibitem[J]{James}
G. D. James, The Representation Theory of the Symmetric Groups.

\bibitem[W]{Woodcock}
D.J. Woodcock, Straightening codeterminants, {\em J. Pure Appl. Algebra} {\bf 88} (1993), 317--320.


\bibitem{BKOP} G. Benkart, S.-J. Kang, S.-J. Oh, and E. Park, Construction of irreducible representations over Khovanov-Lauda-Rouquier algebras of finite classical type, {\em Int. Math. Res. Not. IMRN}, 2014, no. 5, 1312--1366.

\bibitem{BDK}
J. Brundan, R. Dipper, and A. Kleshchev, {\em Quantum linear groups and representations of $GL_n({\mathbb F}_q)$}, Mem.  Amer. Math. Soc. {\bf 149} (2001), No 706, American Mathematical Society, Providence, Rhode Island, 2001.

\bibitem{BKcyc}
J. Brundan, A. Kleshchev, Blocks of cyclotomic Hecke algebras and Khovanov-Lauda algebras, {\em Invent. Math.} {\bf 178} (2009), 451--484

\bibitem{BKOld} J. Brundan and A. Kleshchev,  Homological properties of finite type Khovanov-Lauda-Rouquier algebras, {\tt arXiv:1210.6900v1} (this is a first version of \cite{BKM}). 

\bibitem{BKM}
J. Brundan, A. Kleshchev, and P.J. McNamara,  
Homological properties of finite type Khovanov-Lauda-Rouquier algebras, {\em Duke Math. J.} {\bf 163} (2014), 1353--1404. 

%\bibitem{BKW} J. Brundan, A. Kleshchev and W. Wang, Graded Specht modules, {\em Journal fur die reine und angewandte Mathematik}, {\bf 655} (2011), 61--87.

\bibitem{CPS}
E. Cline, B. Parshall and L. Scott, Stratifying endomorphism Algebras, {\em Mem. Amer. Math. Soc.} {\bf 124} (1996), no. 591.

\bibitem[DKR]{DKR}
J. Desarmenien, J. P. S. Kung, and G. C. Rota, Invariant theory, Young bitableaux, and combinatorics, {\em Adv. in Math.} {\bf 27} (1978), 63--92.

\bibitem{Green}
J.A. Green, {\em Polynomial representations of $GL_n$}, 2nd edition, Springer-Verlag, Berlin, 2007. 

\bibitem[G]{Green}
J. A. Green, Combinatorics and the Schur algebra, {\em J. Pure Appl. Algebra} {\bf 88} (1993), 89--106.

\bibitem{HK} R.S. Huerfano and M. Khovanov, A category for the adjoint representation, {\em J. Algebra} {\bf 246} (2001), 514--542.

\bibitem%[Kac]
{Kac} V. G. Kac, {\em Infinite Dimensional Lie Algebras}, Cambridge University Press, Cambridge, 1990. 

\bibitem{KKK}
S.-J. Kang, M. Kashiwara and M. Kim, Symmetric quiver Hecke algebras and $R$-matrices of quantum affine algebras, {\tt arXiv:1304.0323}. 

\bibitem{Kato}
S. Kato, Poincar\'e-Birkhoff-Witt bases and Khovanov-Lauda-Rouquier algebras, {\em Duke Math. J.} {\bf 163} (2014), 619--663.

\bibitem{KL1}
M. Khovanov and A. Lauda,
A diagrammatic approach to categorification of quantum
groups I, {\em Represent. Theory} {\bf 13} (2009), 309--347. 

\bibitem{Kcusp}
 A. Kleshchev, Cuspidal systems for affine Khovanov-Lauda-Rouquier algebras, {\em Math. Z.}, {\bf 276} (2014),  691--726. 
 
\bibitem{Kdonkin}
A. Kleshchev, Affine highest weight categories and affine quasihereditary algebras, {\em Proc. Lond. Math. Soc. (3)} {\bf 110} (2015),  841--882.

\bibitem{Ksing}
 A. Kleshchev, Representation theory and cohomology of Khovanov-Lauda-Rouquier algebras, 
pp.  109-164 in {\em "Modular Representation Theory of Finite and p-Adic Groups"}, Institute for Mathematical Sciences Lecture Notes Series, Volume 30, World Scientific, 2015. 
% {\tt arXiv:1401.6151}. 

\bibitem{KM}
 A. Kleshchev and R. Muth, Imaginary Schur-Weyl duality, {\em Mem. Amer. Math. Soc.}, to appear; {\tt  arXiv:1312.6104}.
 

\bibitem{KMZZ}
 A. Kleshchev and R. Muth, Affine zigzag algebras and imaginary strata for KLR algebras, preprint, University of Oregon, 2015. 
 

 
 \bibitem{KRhomog}
 A. Kleshchev and A. Ram, Homogeneous representations of Khovanov-Lauda algebras, {\em J. Eur. Math. Soc.} {\bf 12} (2010), 1293--1306.
   
\bibitem{KRbz}
A. Kleshchev and A. Ram, Representations of Khovanov-Lauda-Rouquier algebras and combinatorics of Lyndon words, {\em Math. Ann.} {\bf 349} (2011), 943--975. 



\bibitem{KS}
 A. Kleshchev and D. Steinberg, Homomorphisms between standard modules over finite type KLR algebras,  {\tt  arXiv:1505.04222}.

\bibitem{LV}
A. Lauda and M. Vazirani, Crystals from categorified quantum groups, {\em Adv. Math.} {\bf 228} (2011), 803--861. 
%{\tt  arXiv:0909.1810}.


\bibitem{McN}
P. McNamara, Finite dimensional representations of Khovanov-Lauda-Rouquier algebras I: finite type, {\em J. Reine Angew. Math.}, to appear; {\tt  arXiv:1207.5860}.

\bibitem{McNAff}
P. McNamara, Representations of Khovanov-Lauda-Rouquier algebras III: symmetric affine type, {\tt ArXiv:1407.7304v2}. 

%\bibitem{PS} B. Parshall and L. Scott, Q-Koszul algebras and three conjectures, {\tt arXiv:1405.4419}.

\bibitem{Ro}
R. Rouquier, $2$-Kac-Moody algebras; {\tt arXiv:0812.5023}. 




\bibitem{Thompson}
J.G. Thompson, Vertices and sources, {\em J. Algebra} {\bf 6} (1967), 1--6.

\bibitem{TW}
P. Tingley and B. Webster, Mirkovic-Vilonen polytopes and Khovanov-Lauda-Rouquier algebras, {\tt arXiv:1210.6921}. 



}\fi



\end{thebibliography}
\end{document}